\documentclass[10pt]{amsart}
\usepackage{fancyhdr}
\usepackage[english]{babel}
\usepackage{amssymb}
\usepackage{mathrsfs}
\usepackage{enumerate}
\usepackage{color}
\usepackage{ifpdf}
\usepackage{tikz}
\usepackage{tikz-cd}
\usetikzlibrary{decorations.pathmorphing,arrows}
\usepackage[initials]{amsrefs}
\usepackage{hyperref}

\usepackage[textsize=footnotesize,textwidth=15ex]{todonotes}

\usepackage{color}

\newcommand{\bbQ}{{\mathbb Q}}
\newcommand{\bbR}{{\mathbb R}}

\newcommand{\bbZ}{{\mathbb Z}}

\newcommand{\SL}{\operatorname{SL}}

\newcommand{\SO}{\operatorname{SO}}

\newcommand{\acts}{\curvearrowright}

\newcommand{\pr}{\operatorname{pr}}

\newcommand{\Isom}{\operatorname{Isom}}

\newcommand{\bdd}{\operatorname{Bdd}}

\newcommand{\Prob}{\operatorname{Prob}}

\newcommand{\half}{\frac{1}{2}}

\newcommand{\diam}{\operatorname{diam}}

\newcommand{\Atm}{\operatorname{Atom}}

\newcommand{\Map}{\operatorname{Map}}

\newcommand{\overto}[1]{{\buildrel{#1}\over\longrightarrow}}

\newcommand{\setdef}[2]{ \left\{ {#1}\ \mid\ {#2} \right\} }

\newtheorem{theorem}{Theorem}[section]
\newtheorem{thm}[theorem]{Theorem}
\newtheorem{lemma}[theorem]{Lemma}
\newtheorem{claim}[theorem]{Claim}

\newtheorem{corollary}[theorem]{Corollary}

\newtheorem{prop}[theorem]{Proposition}

\theoremstyle{definition}
\newtheorem{definition}[theorem]{Definition}

\newtheorem{example}[theorem]{Example}

\newtheorem{remark}[theorem]{Remark}

\newtheorem*{ques*}{Question}

\numberwithin{equation}{section}

\begin{document}

\title[Hyperbolic actions of lattices]{Hyperbolic actions of higher-rank lattices\\ come from rank-one factors}

\author[U. Bader]{Uri Bader}
\address{Weizmann Institute of Science, Israel}
\email{uri.bader@gmail.com}

\author[P.-E. Caprace]{Pierre-Emmanuel Caprace}
\address{Institut de Recherche en Math\'ematiques et Physique, UCLouvain, Belgium}
\email{pe.caprace@uclouvain.be}

\author[A. Furman]{Alex Furman}
\address{University of Illinois at Chicago, Chicago, IL, USA}
\email{furman@uic.edu}

\author[A. Sisto]{Alessandro Sisto}
	\address{Maxwell Institute and Department of Mathematics, Heriot-Watt University, Edinburgh, UK}
	\email{a.sisto@hw.ac.uk}

\date{\today}

\begin{abstract}
We study actions of higher rank lattices $\Gamma<G$ on hyperbolic spaces, and we show that all such actions satisfying mild properties come from the rank-one factors of $G$. In particular, all non-elementary isometric actions on an unbounded hyperbolic space are of this type.  
\end{abstract}

\maketitle


\section{Introduction}


How can a given group act by isometries on a hyperbolic space? The aim of this paper is to study this question for irreducible lattices in a semisimple Lie group $G$ of rank~$\geq 2$. Thomas Haettel \cite{Haettel} addressed the case where all simple factors of the ambient product $G$ have rank~$\geq 2$ and showed, in that case, that the isometric actions of the lattice on hyperbolic spaces are all degenerate (see below for a more precise formulation). In this paper, we allow $G$ to have  simple factors of rank one. Since rank-one  simple groups have a natural geometric action on a proper hyperbolic space (namely, a rank-one symmetric space), the lattices in $G$ do admit non-degenerate actions on hyperbolic spaces via their projections on the  rank-one simple factors of $G$. We show that, up to a natural equivalence, those are the only actions of lattices in $G$ on hyperbolic spaces.  

\subsection{Generalities on actions on hyperbolic spaces} 

Before stating our main theorem, we now explain some general facts about actions on hyperbolic spaces. First of all, any group has isometric actions on hyperbolic spaces that fix a bounded set, as well as actions that fix a point at infinity. Such actions can therefore not be used to deduce anything about the group: from our viewpoint, they are degenerate, and we will disregard them. Moreover, given an action on a hyperbolic space, one could make a larger hyperbolic space containing the first one as a quasi-convex subspace, maintaining the group action. This can be done, for example, by attaching equivariantly geodesic rays. To take this possibility into account, it is natural to also rule out actions that admit a quasi-convex invariant set that is not coarsely dense. In view of all this we define \emph{coarsely minimal} actions (Definition \ref{defn:coarsely_min}) by, essentially, ruling out the pathological behaviours discussed above. Arguably, those are the most general actions that one might want to classify. Moreover, actions on hyperbolic spaces that admit an equivariant quasi-isometry should be considered equivalent, and we capture this in Definition \ref{defn:equivalence}, where there is a subtlety to deal with actions where given subgroups fix a bounded set rather than single points.

\subsection{Rigidity of hyperbolic actions} 

We say that a lattice $\Gamma$ in a product group $G = G_1 \times \dots G_N$ is  \textbf{irreducible} if the closure of the projection of $\Gamma$ to any proper sub-product has finite index. It follows from results of G.~Margulis \cite[Theorem~II.6.7]{Margulis-book} that if each $G_i$ is a simple algebraic group over a local field, then $\Gamma$ is irreducible if and only if $\Gamma$ has no finite index subgroup that splits as a direct product in a nontrivial way. Here is the main result of this paper. 

\begin{theorem}\label{T:main_intro}
Let $N \geq n \geq 0$ be integers. Let $G=\prod_{i=1}^N G_i$ be a product of $N$ centerless simple Lie groups, 
	where for all $i \in \{1, \dots, n\}$, the factor $G_i$ has real rank-one,  and for all $j \in \{n+1, \dots, N\}$, the factor $G_j$ has real rank at least two. Let $\Gamma < G$ be a lattice. 
	Assume that $n\ge 2$ or that $N>n$. If $N>1$, assume in addition that $\Gamma$ is irreducible. 

	Then any coarsely minimal isometric action of $\Gamma$ on a geodesic hyperbolic space is equivalent to one of the actions
	\[
		\Gamma\ \overto{}\ G\ \overto{\pr_i}\ G_i\ \overto{}\ \Isom(X_i,d_i)\qquad (1\le i\le n) 
	\]
	where each $X_i$ is the rank-one symmetric space of 
	the factor $G_i$.
\end{theorem}

As mentioned above, we refer to Definition~\ref{defn:equivalence} for the precise notion of \textit{equivalence} appearing in the theorem.

In the special case where $G$ consists only of a single higher rank factor, that is the case where $n=0$ and $N=1$,
our considerations recover for Lie groups the main theorem of \cite{Haettel}. 

\begin{corollary}\label{cor-intro:NoAction}
	Let $G$ be a simple Lie group of real rank~$\geq 2$,  and $\Gamma<G$ a lattice.
	Then $\Gamma$ does not admit any coarsely minimal isometric action on a geodesic hyperbolic space.
\end{corollary}

One can also view Theorem~\ref{T:main_intro} as a generalization of the main result from Margulis' paper \cite{Margulis-amalgam},
where he studied possible amalgam decompositions of lattices in higher rank. 

\subsection{Hyperbolic structures}

The setup adopted here is inspired by the notion of \textit{hyperbolic structures}, defined in \cite{ABO} to capture 
\emph{cobounded} actions on hyperbolic spaces. Coarsely minimal actions provide a similar but broader setup  (see \cite[Proposition 3.12]{ABO} for a comparison). 
In what follows, we regard a hyperbolic structure as an equivalence class (in the sense of Definition \ref{defn:equivalence}) of cobounded actions on hyperbolic spaces. It is a general fact that any such action is either coarsely minimal, or the hyperbolic space being acted on is bounded (giving rise to what is called the trivial structure). While a coarsely minimal action may fail to be cobounded in general,   it is worth noting that Theorem~\ref{T:main_intro} yields another rigidity feature of higher rank lattices, namely:

\begin{corollary}\label{cor-intro:CoarselyMinimal->cobounded}
	Let $\Gamma < G$ be as in Theorem~\ref{T:main_intro}. 
	Then any coarsely minimal isometric $\Gamma$-action on a geodesic hyperbolic space is cobounded.
\end{corollary}

Indeed, since the projection $\Gamma\to G_i$ has dense image, the result follows from Theorem~\ref{T:main_intro} because the $G_i$-action on the model space $X_i$ is cocompact. 

By Corollary~\ref{cor-intro:CoarselyMinimal->cobounded}, the number of hyperbolic structures up to equivalence on $\Gamma$  is the number of coarsely minimal $\Gamma$-actions up to equivalence plus one. Therefore, in the language of \cite{ABO}, Theorem~\ref{T:main_intro} implies that the lattices under consideration 
have exactly $n+1$ inequivalent hyperbolic structures.
Note that in \cite{ABO}, 
for every integer $n \geq 1$ the authors 
construct a finitely generated group $\Gamma$ admitting precisely $n$ distinct 
hyperbolic structures; irreducible lattices in higher-rank semisimple Lie groups provide 
naturally occurring examples of that same phenomenon. Note that any lattice in a higher-rank simple Lie group has only the trivial hyperbolic structure by  \cite{Haettel} or Corollary~\ref{cor-intro:NoAction}.

\begin{example}[Groups with  $n \geq 2$   non-trivial hyperbolic structures]
\hfill{}\\
Let $n \geq 2$ be an integer. Let also $K$ be a totally real number field of degree $n$. Such a field can be constructed as a subextension of degree $n$ in $\bbQ\big(\cos(2\pi/p)\big)$, where $p$ is any prime congruent to $1$ modulo $2n$ (in the special case $n =2$, one can of course simply take $K = \bbQ(\sqrt 2)$). Let $\mathcal O$ be the ring of integers of $K$. Then $K$ has exactly $n$  inequivalent embeddings as subfields of $\bbR$, whose restrictions to $\mathcal O$ have dense image. The corresponding diagonal embedding of $\mathcal O$ in $\bbR^n$ has discrete image.   Consequently, the group 
$$\Gamma = \SL_2(\mathcal O)$$
embeds as an irreducible lattice in the product of $n$ copies of $\SL_2(\bbR)$. Theorem~\ref{T:main} implies that it has precisely $n$ non-trivial hyperbolic structures that arise from its actions on the hyperbolic plane via the projections $\pr_1,\dots,\pr_n$.
\end{example}

\begin{example}[A group with a single non-trivial hyperbolic structure]\hfill{}\\
	Consider the quadratic form $q(x_1,\dots,x_5)=x_1^2+x_2^2+x_3^2+\sqrt{2}x_4^2-x_5^2$, 
	its orthogonal group $\SO(q)=\setdef{g\in\SL_5}{q\circ g=q}$,
	and let $\Gamma=\SO(q)_{\bbZ[\sqrt{2}]}$ be the group of its integer points.
	This group has only one non-trivial hyperbolic structure, because $\Gamma$ is an irreducible lattice in the semisimple real 
	Lie group ${\SO}(4,1)\times \SO(3,2)$
	that has a single rank-one factor ${\SO}(4,1)\simeq \Isom(\mathbb H^4)$ and another simple $\SO(3,2)$ factor of rank two.
\end{example}
%

\subsection{Outline of proofs}

In the proofs we will use boundary theory as outlined in \cite{BF-ICM}. 
Roughly, given a group $\Gamma$ one can associate to it a Lebesgue space $B$, 
called $\Gamma$-boundary, which on one hand has very strong ergodic properties, 
and on the other hand has the property that whenever $\Gamma$ acts on a compact space $Z$, 
there is a $\Gamma$-equivariant map $B\longrightarrow  \Prob(Z)$, called a \textit{boundary map}. In the case of a group $\Gamma$ acting nicely on a hyperbolic space $X$ as in Theorem~\ref{T:main_intro}, a key step consists in constructing an equivariant measurable map $\Phi:B\overto{}\partial X$ into the Gromov boundary of $X$.  
%
%
%
%
To do so, one starts from a boundary map into $\Prob(\bar X)$ --- the space of probability measures
on a certain compactification of $X$ --- and uses the ergodic properties of the $\Gamma$-action on $B$
in conjunction with the geometric properties for the $\Gamma$-action on $X$,
to show that the map $\Phi$ takes values in Dirac measures $\Phi(x)=\delta_{\phi(x)}$
where $\phi(x)\in \partial X\subset\bar{X}$.
The case where $X$ is a proper Gromov-hyperbolic space had been considered already in \cite{BF-ICM}, 
and indeed the main result of Section~\ref{sec:boundary_maps} is a direct generalization of \cite[Theorem 3.2]{BF-ICM}. 

The fact that $X$ is potentially a non-proper space causes several issues. One of them requires to study the relation between  the horofunction compactification $\bar X$ and the Gromov boundary $\partial X$. Similar considerations appear in the work by B.~Duchesne~\cite{Du} and by Maher and Tiozzo~\cite{MT}. 
Related ideas appeared already in the much earlier work \cite{Margulis-amalgam} of
Margulis, where he studied actions of higher rank lattices on trees of possibly infinite valency. Another issue is that the space $\bdd_C(X)$ of $C$-bounded subsets of $X$, which is naturally quasi-isometric to $X$, need no longer be separable. This causes critical separability issues when dealing with boundary maps. In order to circumvent those difficulties, we introduce in Section~\ref{sec:CoarseMetricErgodicity} a new type of ergodicity of boundary $\Gamma$-spaces, called \textit{coarse metric ergodicity}. We show in  Theorem~\ref{thm:CME} that it  is always satisfied by  Furstenberg--Poisson boundaries. In the context of Theorem~\ref{T:main_intro}, we also show that the standard Furstenberg--Poisson boundary of $G$  also enjoys the coarse metric ergodicity as a $\Gamma$-space (Proposition~\ref{prop:G_P_coarse_ergodic}). In Section~\ref{sec:boundary_maps}, we construct boundary maps for any non-degenerate isometric $\Gamma$-action on a hyperbolic space $X$, under the sole hypothesis of existence of an amenable coarsely metrically ergodic $\Gamma$-space (Theorem~\ref{thm:boundarymap}).

In Section \ref{sec:proof_main}, we focus again on the case where $\Gamma$ is a   higher rank lattice, as in our main theorem. In this case the $\Gamma$-boundary of $\Gamma$ splits as a product, with factors corresponding to the factors of the ambient locally compact group $G$. Due to the ergodicity properties of $\Gamma$-boundaries, we see that when $\Gamma$ acts nicely on a hyperbolic space $X$, the boundary map from the $\Gamma$-boundary to $\partial X$ factors through one of the algebraic factors of the ambient group $G$. At this point, there are two cases to analyse. The first case is when the said factor is of rank~$\geq 2$:  we have to show that this cannot occur. This is done in Subsection \ref{sub:step_2.5_rank_1_factor} by adapting the Weyl group method of Bader--Furman \cite{BF-Weyl}. The second case is that the factor as above corresponds to a rank-one factor $G_i$ of $G$. In that case,  we have to show that $X$ is equivalent to  the model symmetric space $X_i$ for $G_i$. This is done in Subsections \ref{sub:step_2_bounded_to_bounded} and \ref{sub:unbounded_in_g_i_is_unbounded}. By hypothesis, the group $G$ has at least two factors in this case, the  projection of $\Gamma$ to $G_i$ has dense image. To show the equivalence between $X_i$ and $X$, metric properties are transferred from $X_i$ to $X$ via the boundary map.  A key  ingredient is that bounded subsets of $X_i$ correspond to precompact subsets of $G_i$, and the latter property  can be rephrased in terms of the boundedness of Radon--Nikodym derivatives for the action on the $G_i$-boundary.

\begin{remark} 
   In a draft version of this paper that was circulated as a preprint, we had adopted a more general set-up, encompassing irreducible lattices in products whose factors were allowed to be simple algebraic groups over local fields, or standard rank-one groups in the sense of \cite{CCMT}. However, the strategy used in that version was hampered by the measurability issues alluded to above, related to the inseparability of the space $\bdd_C(X)$. This is resolved here using coarse metric ergodicity. It is likely that Theorem~\ref{T:main_intro} nevertheless extends to the general set-up we had originally envisaged. A possible approach to establish this would be to generalize Proposition~\ref{prop:G_P_coarse_ergodic} beyond the Lie group case. 
\end{remark}

\subsection*{Acknowledgements}

The authors would like to acknowledge the support:
UB was supported by ISF Moked 713510 grant number 2919/19, PEC was supported by the FWO and the F.R.S.-FNRS under the EOS programme (project ID 40007542), 
AF was supported in part by NSF Grant DMS 2005493,
UB and AF were supported by BSF Grant 2018258.

This project was started at the Oberwolfach workshop ``Geometric structures in group theory" number 1726. The authors are thankful to the Oberwolfach Institute for hosting the workshop, and to the organizers.

The authors are grateful to the referee for valuable comments and suggestions. 


\section{Coarse metric ergodicity}\label{sec:CoarseMetricErgodicity}

\subsection{Definitions}
The goal of this section is to introduce a coarse version of metric ergodicity. Let us recall some definitions before recalling metric ergodicity and discussing why we need a different version.

Let $G$ be a second countable locally compact  group (abbreviated \textit{lcsc group}).
This includes the case of a countable discrete group $\Gamma$.
A \textit{Lebesgue $G$-space} is a Lebesgue space $(\Omega,\mu)$
with a measurable, measure class preserving action map $G\times \Omega\longrightarrow \Omega$.
A \textit{Borel $G$-space $V$} is a standard Borel space $V$ with a Borel
action map $G\times V\longrightarrow  V$.
By a \textit{Lebesgue map} $\Omega\to V$ we mean an equivalence class of functions from $\Omega$ to $V$, defined almost everywhere and up to an almost everywhere identification, which has a Borel measurable representative.

The Lebesgue $G$-space $\Omega$ is \textit{metrically ergodic} if, given any separable metric space $S$ and continuous homomorphism $\pi\colon G\to \Isom(S)$, the only $G$-equivariant measurable maps $\Omega\to S$ are constant. (Ergodicity in the usual sense corresponds to taking $S=\{0,1\}$ and $\pi$ trivial.)

Roughly, the relevance of metric ergodicity in our setting is that, given a higher rank lattice $\Gamma$ acting on a hyperbolic space $X$, we will need to rule out the existence of certain $\Gamma$-equivariant maps $B\to\partial X^{(3)}$, where $B$ is a $\Gamma$-boundary and $\partial X^{(3)}$ is the space of distinct triples of boundary points of $X$. We are not aware of a natural way of endowing $\partial X^{(3)}$ with a \emph{separable} metric. Instead, we can endow it with a ``coarsely separable'' metric, and we will apply coarse metric ergodicity to this metric. Since it does not cost any additional work, we will actually deal with coarse metrics in the following sense:

\begin{definition}
    A \emph{coarse metric} on a set $U$ is a function $d:U\times U\to [0,\infty)$
such that there exists $C>0$ such that for all $x,y,z \in U$,
\[d(x,x)\leq C, \quad d(x,y) \leq d(y,x)+C, \quad \mbox{and} \quad d(x,z)\leq d(x,y) +d(y,z)+C. \]
The function $d$ is a \emph{$C$-coarse metric} on $U$.

Two coarse metrics are coarsely equivalent if their difference is uniformly bounded.
\end{definition}

\begin{remark} \label{rem:eqmetric}
Every $C$-coarse metric $d$ lies within bounded distance of a genuine metric $d'$. Indeed, one can start with $d$, change it so that $d(x,x)=0$, and then set $d_1(x,y)=\max\{d(x,y),d(y,x)\}$. At this point, it is readily checked that $d_1$ is again a $C$-coarse metric, and to get $d'$ one can add to it $C$ times the discrete $\{0,1\}$-metric.
\end{remark}

\begin{definition}
Let $G$ be an lcsc group.
A Lebesgue $G$-space $\Omega$ is said to be \emph{coarsely metrically ergodic} 
if every Borel coarse metric $d$, defined on some conull subset $\Omega_0\subseteq\Omega$ and $G$-invariant as a Lebesgue map $\Omega\times \Omega\to [0,\infty)$, is essentially bounded, 
that is, there exists $R>0$ such for a.e. pair $(x,x')\in\Omega_0\times\Omega_0$ we have $d(x,x')\leq R$.
Equivalently, by Remark~\ref{rem:eqmetric}, 
if every $G$-invariant Borel metric on any $\Omega_0$ as above is essentially bounded.
\end{definition}

\begin{example}\label{ex:double-ergodic}
    If $B$ is a \textit{doubly ergodic} $G$-space, that is, the diagonal $G$-action on $B\times B$ is ergodic, then $B$ is coarsely metrically ergodic. This follows easily from the fact that given a $G$-invariant Borel coarse metric $d$ on $B$, the subsets $\Delta_R=\{(b,b'): d(b,b')\leq R\}\subseteq B\times B$ are $G$-invariant, and their union is the whole $B\times B$. Hence, one of them needs to have positive measure, hence full measure by ergodicity.
\end{example}

\begin{remark} \label{rem:erg}
    Coarse metric ergodicity does not imply ergodicity, as any finite space with a trivial action is clearly coarsely metrically ergodic.
    However, it is easy to see that any coarsely metrically ergodic space has a finite number of ergodic components.
\end{remark}

Note that in the definition of coarse metric ergodicity we do not assume separability of $d$, merely that it is Borel.
An example to keep in mind is the discrete $\{0,1\}$-metric, which is Borel, but not separable unless $\Omega$ is countable.
A coarse metric space is \emph{coarsely separable} if it is covered by a countable collection of balls of a fixed radius.
By the following remark,  essentially all coarse metrics of concern to us are coarsely separable.

\begin{remark} \label{rem:CS}
Given an lcsc group $G$, every $G$-invariant coarse metric on an ergodic space $\Omega$ is coarsely separable by the following argument.
The space is an exhaustion of countably many concentric balls of growing radii, so there exists a radius $r$ and a non-null $r$-ball.
Translating this ball by a dense countable subgroup of $G$ (which acts ergodically, due to the weak* continuity of the $G$-representation on $L^\infty(\Omega)$), we find a countable collection of non-null $r$-balls whose union is of full measure.
\end{remark}

\subsection{Examples}
A first source of coarsely  metrically ergodic actions of a group $\Gamma$ is given by doubly ergodic $\Gamma$-spaces, as pointed out in Example~\ref{ex:double-ergodic}. For our purposes, we need a amenable $\Gamma$-spaces $B_-$ and $B_+$ such that the $\Gamma$-action on $B_-\times B_+$ is coarsely metrically ergodic, see Theorem~\ref{thm:boundarymap} below. Our main source for such $\Gamma$-spaces is provided by the following. 

\begin{thm} \label{thm:CME}
Let $G$ be an lcsc group.
Let $\mu$ be a spread-out generating probability measure on $G$ and let $B_-$ and $B_+$ be the corresponding past and future Furstenberg-Poisson boundaries, that is the Furstenberg-Poisson boundaries associated with the measures $\check{\mu}$ and $\mu$ correspondingly.
Then $B_-$ and $B_+$ are amenable Lebesgue $G$-spaces such that the diagonal $G$-action on $B_-\times B_+$ is metrically ergodic and coarsely metrically ergodic.
\end{thm}

In the case where  $\mu$ is a symmetric measure, we obtain a single amenable $G$-space $B$, equal to $B_-$ and $B_+$. We emphasize that  the coarse metric ergodicity of the $G$-action on $B \times B$ should not be expected to hold for the same reason as in Example~\ref{ex:double-ergodic}, since  the $G$-action on $B \times B \times B \times B$ need not be ergodic: indeed, in the case where $G$ is a rank one simple Lie group, the space $B$ can be identified with the visual boundary of the symmetric space of $G$ and $B \times B \times B \times B$ is the space of $4$-tuple of distinct points. The cross ratio is a $G$-invariant function on that space.

Instead, the proof of Theorem~\ref{thm:CME} relies on the following lemma.

\begin{lemma}[cf. \cite{BF-ICM}*{Lemma~2.8}] \label{lem:BF2.8}
Let $G$ be an lcsc group and let $\mu$ be a spread-out generating probability measure on $G$.
Let $(B_-,\nu_-)$ and $(B_+,\nu_+)$ be the corresponding past and future Furstenberg-Poisson boundaries.
Given a positive measure subset $E\subset B_-\times B_+$ and $\epsilon>0$, there is $g\in G$ and a positive measure subset 
$E_1^-\subset E^-\cap g^{-1}E^-$, where $E^{-}$ is the projection of $E$ to $B_-$, so that for every $b_-\in E_1^-$, $\nu_+(gE_{b_-})>1-\epsilon$, where $E_{b_-}=\{b_+ \in B_+ \mid (b_-,b_+)\in E)\}$.
\end{lemma}

\begin{proof}[Proof of Theorem~\ref{thm:CME}]
The fact that $B_-$ and $B_+$ are amenable Lebesgue $G$-spaces such that the diagonal $G$-action on $B_-\times B_+$ is metrically ergodic is proven in \cite[Theorem 2.7]{BF-ICM}.
We will show that diagonal $G$-action on $B_-\times B_+$ is coarsely metrically ergodic.

    Fix a 
    $G$-invariant Borel metric $d:(B_-\times B_+)^2\to [0,\infty)$.
    By Remark~\ref{rem:CS} and upon normalization, we assume as we may that $d$ is $1$-separable.
We will show that for a.e. $b_-\in B_-$ the essential diameter of $\{b_-\}\times B_+$ is bounded by $3$.
    By symmetry, 
    for a.e. $b_+\in B_+$ the essential diameter of $B_-\times \{b_+\}$ is bounded by $3$, thus the essential diameter of $B_-\times B_+$ is bounded by $6$.
    
We find a countable collection $A_i$ of non-null $1$-balls whose union is of full measure.
We consider the projection $B_-\times B_+\to B_-$ and for every $i$ we let $A_i^-$ be the image of $A_i$. 
For every natural $n$, we let $A_{i,n}^-$ be the set of those $b_-\in A_i^-$ such that  
\[ \nu_+(\{b_+\in B_+\mid (b_-,b_+)\in A_i\}) \geq 1/n. \]
Note that there exists $N_i$ such that the sets $A_{i,n}^-$ are non-null for every $n\geq N_i$.
Using the symbol $\circeq$ for equation up to null-sets, we get $A_{i}^-\circeq\bigcup_{n\geq N_i}~ A_{i,n}^-$.
We let $A_{i,n}$ be the preimage of $A_{i,n}^-$ in $A_i$ and get
\[ A_{i}\circeq\bigcup_{n\geq N_i}~ A_{i,n} \quad \mbox{and} \quad B_+\times B_-\circeq \bigcup_i~ A_{i}\circeq\bigcup_{i,{n\geq N_i}}~ A_{i,n}. \]
By modifying the sets $A_{i,n}$, we assume as we may that for every $i',n',i'',n''$, $A_{i',n'}^-\cap A_{i'',n''}^-$ is either non-null or empty, upon throwing away the preimages of the intersection in case it is null.
By Fubini, for a.e. $b_-\in B_-$, for a.e. $b_+\in B_+$, $(b_-,b_+)\in \bigcup_{i,{n\geq N_i}}~ A_{i,n}$.
We fix a generic $b_-\in B_-$ and consider the full measure subset of the fiber over $b_-$,
\[ D=(\{b_-\}\times B_+)\cap (\bigcup_{i,{n\geq N_i}}~ A_{i,n}). \]
We claim that the set $D$ is bounded by $3$. In order to see this, let us fix $b_+',b_+''\in B_+$ such that $(b_-,b_+'),(b_-,b_+'')\in D$.
Fix $i',n',i'',n''$ such that $(b_-,b_+')\in A_{i',n'}$, $(b_-,b_+'')\in A_{i'',n''}$ and $n'>N_{i'}$, $n''>N_{i''}$.
Set 
$$\epsilon=\min\{1/n',1/n''\} 
\qquad \text{and} \qquad
E^-=A_{i',n'}^-\cap A_{i'',n''}^-.$$
Let $E$ be the preimage of $E^-$ in $A_{i',n'}$ and note that it is non-null.
By Lemma~\ref{lem:BF2.8}, there exist $g\in G$ and a non-null subset $E_1^-$ in $E^-\cap g^{-1}E^-$ such that for every $b_-'\in E_1^-$, 
\[ \nu_+(gE_{b_-'}) > 1-\epsilon, \]
where $E_{b_-'}=\{b_+\mid (b_-',b_+)\in E\}$.
Fixing such $b_-'$, we conclude that the set $gE_{b_-'}$ is bounded by $1$ and it intersects both sets $A_{i',n'}$ and $A_{i'',n''}$.
We conclude that $d((b_-,b_+'),(b_-,b_+''))<3$.
This finishes the proof.
\end{proof}

For lattices in a semisimple Lie group $G$, one can realize the unique $G$-invariant measure class on $G/P$ as a Furstenberg-Poisson boundary, leading to:

\begin{prop}
\label{prop:G_P_coarse_ergodic}
  Let $G$ be a semisimple Lie group and $\Gamma<G$ a lattice.
  Let $P<G$ be a minimal parabolic and endow the coset space $G/P$ with the unique $G$-invariant measure class and consider it as a Lebesgue $\Gamma$-space.
  Then $G/P$ is an amenable Lebesgue $\Gamma$-space such that the diagonal $\Gamma$-action on $G/P\times G/P$ is metrically ergodic and coarsely metrically ergodic.
\end{prop}

\begin{proof}
We fix a maximal compact subgroup $K<G$.
In \cite{Furst63} Furstenberg proved that all bounded Harmonic functions on the symmetric space $G/K$ are represented via a Poisson formula as bounded measurable functions on $G/P$ with respect to a certain $G$-quasi invariant measure $\nu\in \Prob(G/P)$, namely the unique $K$-invariant measure.
Following the discretization process applied by Furstenberg in the special case of the hyperbolic plane, Lyons and Sullivan obtained in \cite[Theorem~5]{Lyon-Sull} a discretization scheme which applies to any *-recurrent discrete subset of a Riemannian manifold of bounded geometry (see \cite[Section 7]{Lyon-Sull} for the definition of *-recurrent),
an example of such setting being the orbit of $\Gamma$ on $G/K$.
Their discretization procedure for the pair $(G/K,\Gamma.K)$ commutes with the symmetries of the pair, thus provides a probability measure $\mu$ on $\Gamma$ for which the bounded $\mu$-harmonic functions are exactly the restriction to $\Gamma$ of the bounded harmonic functions on $G/K$.
That is, the future Furstenberg-Poisson boundary $B_+$ of $(\Gamma,\mu)$ is $(G/P,\nu)$.
The main theorem of \cite{Bal-Led} shows further that the measure $\mu$ could be chosen to be symmetric, that is $\check{\mu}=\mu$.
This shows that the past Furstenberg-Poisson boundary $B_-$ of $(\Gamma,\mu)$ is $(G/P,\nu)$ as well.
We are thus done by Theorem~\ref{thm:CME}.
\end{proof}

\subsection{Coarse invariance}

One can also make a further coarsification:

\begin{definition}
    A coarse metric on a $G$-space $U$ is said to be coarsely invariant if there exists $C>0$ such that for every $x,y\in U$ and $g\in G$, $|d(gx,gy)-d(x,y)|<C$.
\end{definition}

This coarsification turns out to not correspond to a more restrictive type of ergodicity:

\begin{lemma} \label{lem:coarsecoarseergodic}
    Assume $\Gamma$ is a countable group and $\Omega$ is a coarsely metrically ergodic Lebesgue $\Gamma$-space.
    Then every coarsely invariant Borel coarse metric on $\Omega$ is bounded.
    In fact, every coarsely invariant Borel function $\Omega\times \Omega \to [0,\infty)$ which is coarsely symmetric and satisfies the coarse triangle inequality is a bounded coarse metric.
\end{lemma}

\begin{proof}
    Given a coarsely invariant coarse metric $d$ on $\Omega$, observe that $d'(x,y)=\sup \{d(gx,gy)\mid g\in \Gamma\}$ is a $\Gamma$-invariant measurable coarse metric. Its boundedness implies the boundedness of $d$.
    The last part of the lemma follows from the fact that the function $x\mapsto d'(x,x)$ is automatically bounded, by Remark~\ref{rem:erg}.
\end{proof}

\begin{remark}
    We could say that $\Omega$ is coarsely coarse metrically ergodic but we honestly think this is too much.
\end{remark}

\subsection{Countable quotients}

The following lemma shows that up to coarse equivalence, every coarsely separable coarse metric is pulled back from a countable quotient.

\begin{lemma} \label{lem:countquotienyt}
    Let $d$ be a coarsely separable coarse metric on $U$.
    Then there exists countable measurable quotient $\pi:U\to V$ and a coarse metric $d'$ on $V$ such that $d$ is coarsely equivalent to the metric $\pi^*d'$, defined by $d(x,y)=d'(\pi(x),\pi(y))$.
\end{lemma}

\begin{proof}
        We choose a countable $C$-dense subset $V$ in $U$, identify it with $\mathbb{N}$ and define $\pi$ inductively, sending points at the $C$-ball around $v\in V$ to $v$, unless their $\pi$ value was already determined. We set $d'=d|_V$.
\end{proof}

\subsection{Spaces of probabilities}

Given a Borel space $U$, we endow the space $\Prob(U)$ with the Borel structure generated by the evaluation maps associated with all bounded Borel functions on $U$. We will need Proposition \ref{prop:maps_to_prob} below to rule out the existence of maps from the (double) boundary of the lattices we will be considering to various probability spaces $\Prob(U)$, including for $U$ the space of distinct boundary triples of a hyperbolic space being acted on.

If $U$ is a standard Borel space, then  $\Prob(U)$ is also a standard Borel space.
If $\pi:U\to V$ is a Borel map then the induced map $\pi_*:\Prob(U)\to \Prob(V)$ is a Borel map.

We assume henceforth that the Borel space $U$ is endowed with   a coarse metric $d$. In that case, we define the function
\[ \bar{d}: \Prob(U) \times \Prob(U) \to [0,\infty) \]
by setting
\[ \bar{d}(\mu_1,\mu_2)=\inf \{r>0\mid \exists \mbox{ a Borel set } A \mbox{ in U},~\diam(A)< r,~\mu_1(A),\mu_2(A)>1/2\}. \]

Note that $\bar{d}(\mu,\mu)$ might be arbitrarily large, so this is not necessarily a coarse metric. However, we have the following.

\begin{lemma} \label{lem:bardborel}
The function $\bar{d}$ is symmetric and it satisfies the coarse triangle inequality: the difference $\bar d(\mu_1,\mu_2)-\bar d(\mu_1,\mu_0)-\bar d(\mu_0,\mu_2)$ is uniformly bounded from above.
The restriction of $\bar{d}$ to $U$ via the identification $x\mapsto \delta_x$ gives the coarse metric $(x,y) \mapsto \max\{d(x,y),d(y,x)\}$ which is at a bounded distance away from $d$.

Moreover, if $U$ is a countable set and $d$ is a Borel coarse metric on $U$ then $\bar{d}$ is a Borel function.
\end{lemma}

\begin{proof}
The first assertions of the lemma are straightforward. For the last assertion under the extra hypothesis that   $U$ is countable, we observe that the infimum in the definition of $\bar d$ could be taken over the countable collection of the finite subsets of $U$. This ensures that $\bar d$ is indeed Borel. 
\end{proof}

\begin{prop}
\label{prop:maps_to_prob}
    Let $\Gamma$ be a   countable group and  $\Omega$ be a coarsely metrically ergodic Lebesgue $\Gamma$-space.
    Let $U$ be a Borel space and consider the Borel structure on $\Prob(U)$ generated by evaluation at bounded Borel functions on $U$. 
    If there exists a Borel $\Gamma$-map from $\Omega$ to $\Prob(U)$, then every $\Gamma$-invariant Borel, coarsely separable coarse metric on $U$ has bounded $\Gamma$-orbits.    
\end{prop}

\begin{proof}
    We assume $f:\Omega \to \Prob(U)$ is a Borel $\Gamma$-map
    and $d$ is a coarsely separable $\Gamma$-invariant Borel coarse metric on $U$ and we argue to show that the $\Gamma$-orbits in $U$ are $d$-bounded.
    Since the restriction of $\bar{d}$ to $U$ is at a bounded distance away from $d$ (see Lemma~\ref{lem:bardborel}), it is enough to show that $\Prob(U)$ contains a $\bar{d}$-bounded subset.

We let $\pi:U\to V$ and $d'$ be as in Lemma~\ref{lem:countquotienyt}, so that the pull back $\pi^*d'$ is coarsely equivalent to $d$.
It follows $\overline{\pi^*d'}$ is coarsely equivalent to $\bar{d}$.
By Lemma~\ref{lem:bardborel}, $\bar{d'}$ is Borel, symmetric and it satisfies the coarse triangle inequality. Hence the same applies also to $\overline{\pi^*d'}=(\pi_*)^*\bar{d'}$.

By Lemma~\ref{lem:coarsecoarseergodic}, $f^*\overline{\pi^*d'}$ is a bounded coarse metric on $\Omega$.
We deduce that $\Prob(U)$ contains a $\overline{\pi^*d'}$-bounded subset. The proof is now complete, as this subset is $\bar{d}$-bounded.
\end{proof}

\subsection{Barycenter maps} \label{sec:bary}
The last result in this section is not about coarse metric ergodicity, but it fits here since we are discussing the spaces $\Prob(U)$. Let $U$ be a Borel space and $d$ a measurable metric on $U$. Roughly, we will construct ``barycenters'' for probability measures, and the property that we will need is roughly that measures with bounded Radon-Nykodim derivatives with respect to each other have approximately the same barycenter.

    We denote by $\bdd(U)$ the space of non-empty $d$-bounded subsets of $U$.
    Given a measure $m\in\Prob(U)$ and $\epsilon\in (0,1/2)$ we define
    \[ r_\epsilon(m)=\inf \{\diam(A) \mid A \mbox{ is a Borel set in } U,~m(A)\geq 1-\epsilon\} \]
and we let $\beta_\epsilon(m)$ be the union of all Borel sets $A$ in $U$ of diameter less than $r_\epsilon(m)+1$ and $m(A)\geq 1-\epsilon$.
Note that since $1-\epsilon<1$, the infimum in the definition of $r_\epsilon$ is taken over a non-empty set and $\beta_\epsilon(m)$ is a non-empty subset of $U$.

\begin{lemma}\label{L:centers}
Let $U$ be a Borel space and $d$ a measurable metric on $U$.
    For every $\epsilon \in (0,1/2)$ and for every $m\in\Prob(U)$,
    $\beta_\epsilon(m)$ is a non-empty $d$-bounded subset of $U$, and the map $\beta_\epsilon:\Prob(U)\to \bdd(U)$, is equivariant under all isometries of $U$. 
		Moreover, given $C\geq 1$, if $m_1,m_2\in \Prob(U)$ 
		are such that for every measurable $E\subseteq U$ we have
		\[ 
			m_2(E)\le C m_1(E) 
		\]
		then
		\[
			\beta_{\epsilon}(m_2)\subseteq N_{R}(\beta_{\epsilon/C}(m_1)),
		\]
		where $R=r_{\epsilon/C}(m_1)+1$.
\end{lemma}

\begin{proof}
By the fact that $1-\epsilon>1/2$, all the Borel subsets $A$ with $m(A)\geq 1-\epsilon$ intersect in pairs, so the diameter of $\beta_\epsilon(m)$ is bounded by $2r_\epsilon(m)+2$. Equivariance is clear.

Let now $C,m_1,m_2$ as in the last statement. 
Note that every 
Borel subset $A$ of $U$ with $m_1(A)\geq 1-\epsilon/C$ satisfies 
$m_2(U\setminus A)\leq Cm_1(U\setminus A) \leq \epsilon$, hence 
$m_2(A)\geq 1-\epsilon$. 
It follows that $r_{\epsilon}(m_2)\leq r_{\epsilon/C}(m_1)$.

If $u\in \beta_{\epsilon}(m_2)$, then by definition $u$ belongs to a Borel set $A$ with $m_2(A)\geq 1-\epsilon$ and diameter less than $r_{\epsilon}(m_2)+1\leq r_{\epsilon/C}(m_1)+1$. Fix any Borel set $B$ of diameter less than $r_{\epsilon/C}(m_1)+1$ and $m_1(B)\geq 1-\epsilon/C$. 
Then $m_2(B)\geq 1-\epsilon$ and we get 
$A\cap B\neq \emptyset$.
In particular, $u$ lies within $r_{\epsilon/C}(m_1)+1$ of $B$, whence of $\beta_{\epsilon/C}(m_1)$.
\end{proof}

\section{Boundary maps}
\label{sec:boundary_maps}

In this section we fix a countable group $\Gamma$ and discuss {\em boundary maps}
associated to various hyperbolic structures on $\Gamma$.
In particular we prove Theorem~\ref{thm:boundarymap} below.
The novel aspect of this theorem is that the Gromov hyperbolic spaces 
it deals with are not assumed to be proper.
Recall that the Gromov boundary of a proper Gromov hyperbolic space 
is compact and the associated action 
is a {\em convergence group} action.
Boundary maps associated with such actions were considered in \cite{BF-ICM}*{Theorem~3.2}.

\medskip

Given a Lebesgue $\Gamma$-space $\Omega$ and a standard Borel $\Gamma$-space $V$, 
we denote by $\Map_\Gamma(\Omega,V)$ the space of equivalence classes of measurable $\Gamma$-equivariant
maps $f:\Omega\to V$, i.e. those measurable maps that satisfy
$f(g.\omega)=g.f(\omega)$ for a.e. $g\in \Gamma$ and a.e. $\omega\in\Omega$. 
Two such maps $f, f':\Omega\to V$ are identified if $f(\omega)=f'(\omega)$ for a.e. $\omega\in\Omega$.
Any such map $f$ is equivalent to $f_0:\Omega\to V$ such that for every $g\in \Gamma$
we have $f_0(g.\omega)=g.f_0(\omega)$ a.e. $\omega\in\Omega$ (\cite[Proposition B.5]{Zimmer}).

\medskip

In the following theorem $(X,d)$ is a separable Gromov hyperbolic space,
$\Gamma$ a countable group
and $\Gamma\longrightarrow \Isom(X,d)$ a homomorphism.
We denote by $\partial X$ the associated Gromov boundary, which is a Polish space on which $\Gamma$ acts by homeomorphisms, let $\partial X^2$ be its square,
and
\[
    \partial X^{(2)}=\setdef{(\xi,\eta)}{\xi\ne \eta \in \partial X}
\]
the subset of  pairs of distinct boundary points. 
Since $X$ is not assumed to be proper, $\partial X$ is not 
necessarily compact. Yet, $\partial X$ is a standard Borel space,
and so are $\partial X^{(2)}\subset \partial X^2$.
The action of $\Gamma$ on all these spaces is Borel.
 
\begin{theorem}[{Compare \cite{BF-ICM}*{Theorem~3.2}}] \label{thm:boundarymap}\hfill{}\\
Let $\Gamma$ be a countable group. Assume $B_-$ and $B_+$ are amenable Lebesgue $\Gamma$-spaces such that the diagonal $\Gamma$-action on $B_-\times B_+$ is ergodic and coarsely metrically ergodic.
	Let $(X,d)$ be a separable, Gromov hyperbolic (possibly non-proper), geodesic metric space
	and let $\Gamma$ act on $X$ by isometries.
	Assume that $\Gamma$ does not fix a bounded set in $X$ and does not fix a point or
	a pair of points in $\partial X$.
	
	Then there exist $\phi_-\in \Map_\Gamma(B_-,\partial X)$, 
	$\phi_+\in\Map_\Gamma(B_+,\partial X)$
	such that the image of the map $\phi_{\bowtie}\in\Map_\Gamma(B_-\times B_+,\partial X^2)$
	given by
	\[
		\phi_{\bowtie}(x,y)=(\phi_-(x),\phi_+(y))
	\]
	is contained, modulo a null set, in the set of distinct pairs 
	$\partial X^{(2)}\subset \partial X^2$. Moreover, each of the following assertions holds.
	\begin{itemize}
		\item[{\rm (i)}] $\Map_\Gamma(B_-,\Prob(\partial X))=\left\{ \delta\circ \phi_-\right\}$,
		and $\Map_\Gamma(B_+,\Prob(\partial X))=\left\{ \delta\circ \phi_+\right\}$.
		\item[{\rm (ii)}] $\Map_\Gamma(B_-\times B_+,\partial X) = \{ \phi_-\circ \pr_-,\phi_+\circ \pr_+\}$,
		\item[{\rm (iii)}] $\Map_\Gamma(B_-\times B_+,\partial X^{(2)}) = \{ \phi_{\bowtie},\ \tau\circ \phi_{\bowtie}\}$, 
		where $\tau(\xi,\xi')=(\xi',\xi)$.
	\end{itemize}
\end{theorem}

The rest of this section is devoted to the proof of this theorem.

\subsection{The horoclosure of a separable metric space}

Let $(X, d)$ be a separable metric space. 
We consider the vector space of functions from $X$ to $\mathbb{R}$
endowed with the topology of pointwise convergence, i.e. the product space $\mathbb{R}^X$,
and the constant function $\mathbf{1}\in \mathbb{R}^X$.
We endow $\mathbb{R}^X/\mathbb{R}\cdot \mathbf{1}$ with the quotient topological vector space structure.
We map $X$ to $\mathbb{R}^X$ by $x\mapsto d(\cdot,x)$
and consider its image in $\mathbb{R}^X/\mathbb{R}\cdot\mathbf{1}$.
We denote the closure of the image of $X$ in $\mathbb{R}^X/\mathbb{R}\cdot \mathbf{1}$ by $\bar{X}$ and call it the
{\em horoclosure} of $X$.
We denote the obvious map $X\to \bar{X}$ by $i$, and the preimage of $\bar{X}$ in $\mathbb{R}^X$ by $\tilde{X}$.
Elements of $\bar{X}$ (and by abuse of notations, also elements of $\tilde{X}$) are called \emph{horofunctions}.
It is a common practice to fix a base point $x\in X$ and to consider the subspace
\[ 
	\tilde{X} \supset \tilde{X}_x=\setdef{h\in \tilde{X}}{ h(x)=0 }.
\]

It is well-known that the horoclosure of $X$ is a compactification of it:

\begin{lemma} \label{lem:horo_inject}
	$\bar{X}$ is a compact metrizable space and the map $i:X\to \bar{X}$ is an injective continuous map.
	For a fixed $x\in X$, the map $\tilde{X}_x\to \bar{X}$ is a homeomorphism.
\end{lemma}

\begin{proof}
	The fact that $i$ is continuous is obvious.
	For $x\neq y$ in $X$, note that the difference function $d(x,\cdot)-d(y,\cdot)$ is not constant, 
	as it attains different values at $x$ and $y$. Thus $i$ is injective.
	We now fix $x\in X$. First we note that $\tilde{X}_x$ is closed subset of
	\[ 
		\prod_{y\in X} [-d(x,y),d(x,y)] \subset \prod_{y\in X} \mathbb{R}=\mathbb{R}^X, 
	\]
	thus it is compact.
	Fixing a countable dense subset $X_0$ in $X$,
	the obvious map $\tilde{X}_x\to \mathbb{R}^X \to \mathbb{R}^{X_0}$ is a continuous injection 
	(as $\tilde{X}_x$ consists of continuous functions), 
	hence a homeomorphism onto its image.
	The image is a Frechet space, thus metrizable.
	It follows that $\tilde{X}_x$ is metrizable.
	Since the natural map $\tilde{X}_x\to \bar{X}$ is also a continuous bijection, 
	we conclude that it is a homeomorphism
	and deduce that $\bar{X}$ is compact and metrizable.
\end{proof}

Loosely speaking, we often identify $X$ with $i(X) \subset \bar{X}$.
Note, however, that the image of $X$ is in general not open in $\bar{X}$, and the map $i$ is not
a homeomorphism onto its image.

We decompose $\tilde{X}$ as follows.
\[ 
	\begin{split}
	\tilde{X}^b &=\setdef{h\in \tilde{X}}{ f \mbox{ is bounded from below}}, \\
	\tilde{X}^u &=\setdef{h\in \tilde{X}}{ f \mbox{ is unbounded from below}}.
	\end{split}
\]
This decomposition is constant on the fibers of $\tilde{X}\to \bar{X}$, 
thus gives a corresponding decomposition $\bar{X}=\bar{X}^b\cup \bar{X}^u$. 
Clearly we have $i(X) \subseteq \bar X^b$, so that $\bar X^b$ is dense in $\bar X$. 

\begin{lemma}
	The decompositions $\tilde{X}=\tilde{X}^b\cup \tilde{X}^u$ and $\bar{X}=\bar{X}^b\cup \bar{X}^u$ are measurable
	and $\Isom(X)$-equivariant.
\end{lemma}

\begin{proof}
	The equivariance of the decompositions is obvious.
	Fix a dense countable subset $X_0$ in $X$ and use the fact that $\tilde{X}$ consists of continuous functions to note that
	\[ 
		\tilde{X}^u=\bigcap_{n\in \mathbb{N}}\bigcup_{x\in X_0} 
		\setdef{h\in \tilde{X}}{ h(x)\leq -n },
	\]
	thus $\tilde{X}^u \subset \tilde{X}$ is measurable.
	Fixing $x\in X$, using the measurability of $\tilde{X}^u_x$, we observe that $\bar{X}^u\subset \bar{X}$ is measurable.
\end{proof}

\subsection{The horoclosure of a hyperbolic metric space}

We now assume, in addition, that the  separable metric space $(X,d)$ is geodesic and Gromov hyperbolic 
(as before, it is  possibly non-proper), and that there is a $\Gamma$-action on it for some group $\Gamma$.

The following lemma will allow us to use coarse metric ergodicity in the setting of Theorem \ref{T:main} to exclude the existence of $\Gamma$-maps to $\bar{X}^b$.

\begin{lemma}\label{lem:horo_to_sets}
There is a Borel coarsely separable pseudo-metric on $\bar{X}^b$ on which $\Gamma$ acts by isometries. Moreover, if the $\Gamma$-action on $X$ is unbounded, then the $\Gamma$-action on $\bar{X}^b$ is unbounded.
\end{lemma}

\begin{proof}
We will show the analogous statement for $\tilde{X}^b$ rather than $\bar{X}^b$, and we note that the pseudo-metric we construct below passes to $\bar{X}^b$ (since the sets $\tilde{I}(h)$ below do not change when adding a constant to $h$).

We first show that the function
\[ 
    \inf:\tilde{X}^b\to \mathbb{R}, \quad h\mapsto \inf\setdef{h(x)}{x\in X} 
\]
is measurable and $\Isom(X)$-invariant, and that for every $h\in \tilde{X}^b$,
the set
\[ 
    \tilde{I}(h)=\setdef{x\in X}{h(x)< \inf(h)+1}
\]
is bounded in $X$.

To see that $\inf$ is measurable, we fix a dense countable subset $X_0$ in $X$ and we use the continuity of the functions in $\tilde{X}^b$ to observe that
\[ 
    \inf(h)=\inf\setdef{h(x)}{x\in X_0}. 
\]
The invariance of this function under $\Isom(X)$ is clear.

Fix $h \in \tilde{X}^b$.
We now argue that $\tilde{I}(h)$ is of diameter bounded by $C=8+4\delta$,
where $\delta$ is the hyperbolicity constant associated with the thin triangle property of $X$.
Without loss of generality we assume that $\inf(h)=0$.
Assuming the negation, we fix two points $x,x'$ satisfying
$d(x,x')>C$ and $h(x),h(x')<1$.
We consider a finite sequence of points $x_0,x_1,\ldots,x_n$ on a geodesic
segment from $x$ to $x'$ such that $x_0=x$, $x_n=x'$ and $d(x_i,x_{i+1})<1$.
We consider the image of $h$ in $\bar{X}$ along with its neighborhood
given by
\[ 
	U=\setdef{f+\mathbb{R}\cdot\mathbf{1}}{ f\in \mathbb{R}^X,~\forall 0\leq i,j\leq n,~
	|\left(f(x_i)-f(x_j)\right)-\left(h(x_i)-h(x_j)\right)|<1}. 
\]
We fix a point $y\in X$ whose image in $\bar{X}$ is in $U$.
We thus have:
\begin{equation} \label{eq:ij}
	\forall 0\leq i,j \leq n, \quad
	|d(y,x_i)-h(x_i)-d(y,x_j)+h(x_j)|<1.
\end{equation}
We consider geodesic segments from $y$ to $x$ and from $y$ to $x'$ and, using that $x,x'$ and $y$ 
are the vertices of a thin triangle, we
fix $i$ such that $x_i$ lies at distance at most $1+\delta$ from these segments.
Thus
\[ 
	\begin{split}
		&d(y,x_i)+d(x_i,x)\leq d(y,x)+2+2\delta,\\
		&d(y,x_i)+d(x_i,x')\leq d(y,x')+2+2\delta.
	\end{split}
\]
Note that $d(x,x_i)+d(x_i,x')=d(x,x')$.
Upon possibly interchanging the roles of $x$ and $x'$, we will assume that $d(x,x_i)\geq d(x,x')/2$.
In particular, $d(x,x_i)\geq C/2$.
Taking $j=0$ in Equation~(\ref{eq:ij}), we now have
\begin{align*}
 	0 &=  \inf(h) \\
 	& \leq h(x_i) \\
 	& < 1+d(y,x_i)+h(x)-d(y,x)  \\
	&\leq   1+\left(d(y,x)+2+2\delta- d(x_i,x)\right)+h(x)-d(y,x)\\
	&< (3+h(x))+2\delta-d(x_i,x) \\
	& \leq  4+2\delta-C/2 =0.
\end{align*}

This is a contradiction, thus indeed the diameter of $\tilde{I}(h)$ is bounded by $C$.

We can now define the required pseudo-metric $\rho$ by setting
$$\rho(h,h')=d_{Haus}(\tilde{I}(h), \tilde{I}(h')).$$
This is the pull-back of a metric, so it is a pseudo-metric, and $\Gamma$ clearly acts on it by isometries. The fact that the $\tilde{I}(\cdot)$ are bounded implies the claim on unboundedness of the action. Moreover, coarse separability follows from separability of $X$ and (uniform) boundedness of the $\tilde{I}(\cdot)$. What is left to show is that $\rho$ is Borel, and in order to do so we have to show that for all $r\in \mathbb R$ we have that $\Delta_r=\{(h,h'): \rho(h,h')> r\}\subseteq \tilde{X}^b \times \tilde{X}^b$ is Borel. Expanding the definitions, $\Delta_r$ is the set of pairs $(h,h')$ such that $\tilde{I}(h)$ and $\tilde{I}(h')$ lie at Hausdorff distance more than $r$. In turn, this means that there exists $x\in \tilde{I}(h)\cap X_0$ such that all $y\in \tilde{I}(h')\cap X_0$ lie at distance more than $r$ from $x$, or that the same holds switching the roles of $h$ and $h'$ (to justify the ``$\cap X_0$'' note that $\tilde{I}(h)\cap X_0$ is dense in $\tilde{I}(h)$ in view of the continuity of $h$, and similarly for $h'$). Denoting $\tilde{J}(z)=\{k\in \tilde{X}^b: z\in \tilde{I}(k)\}$ for $z\in Z$, the first case is further equivalent to $h'$ not lying in $\bigcup_{y\in X_0 : d(x,y)\leq r} \tilde{J}(y)$, which is the subset of $\tilde{X}^b$ of all $k$ such that $\tilde{I}(k)$ contains some $y\in X_0$ within distance $r$ of $x$ (the other case is similar).
Hence, we have
$$\Delta_r=\bigcup_{x\in X_0}\left( \tilde{J}(x) \times \left(\tilde{X}^b-\bigcup_{y\in X_0 : d(x,y)\leq r}\tilde{J}(y)\right) \right)\cup$$
$$\bigcup_{y\in X_0} \left( \left(\tilde{X}^b-\bigcup_{x\in X_0 : d(x,y)\leq r} \tilde{J}(y)\right)\times \tilde{J}(x)\right).$$
Hence, $\Delta_r$ is Borel provided that all $\tilde J(z)$ are Borel. But $\tilde J(z)=\{k\in \tilde{X}^b: k(z)<\inf(k)+1\}$, so it is measurable in view of the measurability of $\inf$, as required.
\end{proof}

\subsection{The Gromov boundary}\label{sec:Gromov-bd}

While the construction of the Gromov boundary $\partial X$ is fairly standard,
it is commonly taken under a properness assumption on $X$.
In preparation for our more general discussion we review this construction (for $X$ possibly non-proper) below, and we also relate it to horoclosures.
As common, we fix from now on a base point $o\in X$.
For $x\in X$ we use the shorthand notation $|x|=d(o,x)$.
Gromov products will be taken, unless otherwise stated, with respect to $o$.
That is, for $x,y\in X$ we set
\[ 
	(x,y)=\half\left(|x|+|y|-d(x,y)\right).
\]
In our discussion below we fix $\delta>0$ such that for every $x,y,z\in X$ we have
\[ 
	(x,z) \geq \min\{(x,y),(y,z)\}-\delta. 
\]
We recall that a sequence of points $(x_n)$ in $X$ is said to {\em converge to infinity} if the real numbers $(x_n,x_m)$ 
approach infinity when both indices $m$ and $n$ tend to infinity.

\begin{lemma}
Assume $(x_n)$ is a sequence of points in $X$ which converges in $\bar{X}$
and denote $\bar{h}=\lim x_n$.
Then $(x_n)$ converges to infinity if and only if $\bar{h}\in \bar{X}^u$.
In that case, if $(x'_n)$ is another sequences in $X$ satisfying
$\lim x'_n =\bar{h}$
then $(x_n,x'_n) \to \infty$.
\end{lemma}

\begin{proof}
We will denote the lift of $\bar{h}$ in $\tilde{X}_o$ by $h$ and show that
$(x_n)$ converges to infinity if and only if $h\in \tilde{X}^u_o$.
Note that for every $x\in X$, we have $d(x_n,x)-|x|\to h(x)$.

Assuming first $(x_n)$ converges to infinity, we will show that $h\in \tilde{X}^u_o$.
Fix $r>0$.
Fix $N$ such that for $n,m>N$, $(x_n,x_m)>r$.
Fix $m>N$, note that $|x_m|\geq r$ and let $x$ be a point on a geodesic segment from $o$ to $x_m$
with $|x|=r$.
Then by hyperbolicity,
\[ (x_n,x) \geq \min\{(x_n,x_m),(x_m,x)\}-\delta = r-\delta, \]
Thus
\[ h(x)=\lim_{n\to \infty} \left(d(x_n,x)-|x_n|\right)=
\lim_{n\to \infty} \left(|x|-2(x_n,x)\right) \leq 2\delta-r. \]
As $r$ was arbitrary,  we get indeed that $h\in \tilde{X}^u_o$.

Assuming now $h\in \tilde{X}^u$, we will show that
$(x_n)$ converges to infinity.
Fix $r>0$. Fix $x$ such that $h(x)<-r$.
Fix $N$ such that for every $n>N$, $d(x_n,x)-|x_n|<-r$
and observe that for such $n$,
\[ 
    (x_n,x)=\half\left(|x_n|+|x|-d(x_n,x)\right)\geq -\half\left(d(x_n,x)-|x_n|\right) 
    > \half r. 
\]
Then by hyperbolicity, for $n,m>N$,
\[ 
    (x_n,x_m) \geq \min\{(x_n,x),(x,x_m)\}-\delta > \half r-\delta.
\]
As $r$ was arbitrary,  we deduce indeed that the sequence $(x_n)$ converges to infinity.

In the setting of the former paragraph,
if $(x'_n)$ is another sequences in $X$ satisfying $x'_n \to \bar{h}$,
fixing $N'\geq N$ such that for every $n>N'$, $d(x'_n,x)-|x'_n|<-r$, the same computation shows that
$(x_n,x'_n) > r/2-\delta$.
Thus indeed, we obtain $(x_n,x'_n) \to \infty$.
\end{proof}

Two sequences which converge to infinity, $(x_n)$ and $(y_n)$, are said to be \textit{equivalent} if $(x_n,y_n) \to \infty$.
We conclude that if $(x_n)$ and $(x'_n)$ are two sequences in $X$ satisfying
\[ 
    \lim_{n\to\infty} x_n = \lim_{n\to\infty} x'_n \in \bar{X}^u 
\]
then $(x_n)$ is equivalent to $(x'_n)$.

A point in $\partial X$ is, by definition, an equivalence class of sequences which converge to infinity.
We denote by $\pi$ the unique
map 
$$\pi:\bar{X}^u\to \partial X$$
satisfying
\[ 
	\lim_{n\to\infty} x_n \in \bar{X}^u \qquad \Longrightarrow 
	\qquad \pi\left(\lim_{n\to\infty} x_n\right)=[x_n]. 
\]
For a point $\xi\in \partial X$ and $r>0$ we set
\[ 
	U(\xi,r)=\setdef{\eta\in \partial X}{ \sup\setdef{
	\liminf_{n\to\infty} (x_n,x'_n)}{(x_n)\in\xi,~(x'_n)\in \eta} \geq r}. 
\]
We note that the collection of sets $U(\xi,r)$ forms a basis for a topology and endow $\partial X$ with this topology. In fact, one can extend the sets $U(\xi,r)$ to $X\cup\partial X$, and define a topology there using the extended sets as bases.

\begin{lemma}
The map $\pi$ is continuous and $\Isom(X)$-equivariant.
\end{lemma}

\begin{proof}
The equivariance of $\pi$ is obvious.
In order to show continuity,
we fix a point $\bar{h}\in \bar{X}^u$ and show the continuity of $\pi$ at $\bar{h}$.
Thus
we fix $r>0$ and argue to show that there exists a neighborhood $V$ of $\bar{h}$ in $\bar{X}^u$ such $\pi(V)\subset U(\pi(\bar{h}),r)$.
We will denote the lift of $\bar{h}$ in $\tilde{X}^u_o$ by $h$,
set $t=2(r+\delta)$
and
fix a point $x\in X$ such that $h(x)<-t$.
We let $V\subset \bar{X}^u_o$ be the open neighborhood of $\bar{h}$ corresponding to the set $\{h'\in \tilde{X}^u_o \mid h'(x)<-t\}$.
Fix $\bar{h}'\in V$ and denote its lift in $\tilde{X}_o^u$ by $h'$.
Let $(x_n)$ and $(x'_n)$ be sequences in $X$ converging to $\bar{h}$ and $\bar{h}'$ respectively.
In particular, we have 
\[
    h(x) =\lim_{n\to\infty} \left(d(x_n,x)-|x_n|\right)
    \qquad \text{and} \qquad
    h'(x) =\lim_{n\to\infty} \left(d(x'_n,x)-|x'_n|\right).
\]
Fix $N$ such that for every $n>N$ both $\left(d(x_n,x)-|x_n|\right)<-t$ and $\left(d(x'_n,x)-|x'_n|\right)<-t$.
Note that for $n>N$
\[ 
    (x_n,x)=\half\left(|x_n|+|x|-d(x_n,x)\right)\geq -\half \left(d(x_n,x)-|x_n|\right) 
    > \half t 
\]
and similarly $(x'_n,x)>t/2$.
Thus we have
\[ 
    (x_n,x'_n)\geq  \min\{(x_n,x),(x,x'_n)\}-\delta > \half t-\delta=r. 
\]
It follows that $\liminf (x_n,x'_n)\geq r$
and in particular,
\[ 
    \sup\setdef{\liminf_{n\to\infty} (x_n,x'_n)}{(x_n)\in\pi(\bar{h}),~(x'_n)\in \pi(\bar{h}')} \geq r. 
\]
Therefore $\pi(\bar{h}')\in U(\pi(\bar{h}),r)$, and 
we conclude that indeed $\pi(V)\subset U(\pi(\bar{h}),r)$.
\end{proof}

Here and in the sequel, we use the term  `\textit{$\Gamma$}-map' to mean a measurable $\Gamma$-equivariant map between spaces equipped with a $\Gamma$-action.

The following lemma will be used, similarly to Lemma \ref{lem:horo_to_sets}, in conjunction with coarse metric ergodicity to exclude the existence of $\Gamma$-maps to $\partial X^{(3)}$.

\begin{lemma}\label{lem:triples_to_sets}
    Let $\Gamma$ be a countable group acting by isometries on the separable hyperbolic space $X$. Then there is a Borel coarsely separable pseudo-metric $\rho$ on $\partial X^{(3)}$ such that the natural $\Gamma$-action on $\partial X^{(3)}$ is an action by isometries of $\rho$. Moreover, if $\mathcal U\subseteq \Gamma$ acts on $X$ with unbounded orbits, then the $\mathcal U$-orbits in $(\partial X^{(3)},\rho)$ are also unbounded.
\end{lemma}

\begin{proof}
    Denote $\mathcal T=\partial X^{(3)}$ for convenience. Fix a $\Gamma$-invariant dense subset $A$ of $X$.

We claim that there exists a $\Gamma$-equivariant map $\tau:\mathcal T\to \bdd(A)$ with the property that for any $x\in X$ we have that $\mathcal T_x=\tau^{-1}(\{B\in\bdd(X): x\in B\})\subseteq \mathcal T$ is closed.

To show the claim, we will in fact first define a closed subset $\mathcal T_x$ for each $x\in A$ and then set $\tau(t)=\{x\in A: t\in \mathcal T_x\}$. Let $\delta>0$ be a hyperbolicity constant for $X$.
For $x\in A$, let $T_x\subseteq X^3$ be the set of all triples $(x_1,x_2,x_3)$ such that $d(x_i,x_j)> d(x_i,x)+d(x,x_j)-10\delta$ for all distinct $i,j\in\{1,2,3\}$.
We then define $\mathcal T_x$ to be the intersection of $\mathcal T$ with closure of $T_x$ (taken in $(X\cup\partial X)^3$).
Notice that $\tau(t)$ defined as above is then a bounded, non-empty subset of $A$, and that $\tau$ is $\Gamma$-equivariant.

We can now pull-back the Hausdorff distance $d_{Haus}$ on $\bdd(A)$ to $\mathcal T$, that is, we set $\rho(t,t')=d_{Haus}(\tau(t),\tau(t'))$. Clearly, $\rho$ is a pseudo-metric and $\Gamma$ acts on $\mathcal T$ by isometries of $\rho$. Moreover, orbits are unbounded since $(\mathcal T,\rho)$ is $\Gamma$-equivariantly isometrically embedded in $\bdd(A)$, and on the latter the $\Gamma$-orbits are unbounded since orbits are unbounded in $X$ (a similar argument applies to subsets of $\Gamma$). Coarse separability also from this isometric embedding together with separability of $X$, which gives coarse separability of $\bdd(A)$. It remains to show that $\rho$ is Borel. In order to do so, it suffices to show that for all $r\in \mathbb R$ we have that $\Delta_r=\{(t,t'): \rho(t,t')> r\}\subseteq \mathcal T \times \mathcal T$ is Borel. Expanding the definitions, $\Delta_r$ is the set of pairs $(t,t')$ such that $\tau(t)$ and $\tau(t')$ lie at Hausdorff distance more than $r$. In turn, this means that there exists $x\in \tau(t)$ such that all $y\in \tau(t')$ lie at distance more than $r$ from $x$, or that the same holds switching the roles of $t$ and $t'$. The first case is further equivalent to $t'$ not lying in $\bigcup_{y\in A : d(x,y)\leq r} \mathcal T_y$, so that $\tau(t)$ does not contain any $y$ with $d(x,y)\leq r$ (and the other case is similar).
Hence, we have
$$\Delta_r=\bigcup_{x\in A}\left( \mathcal T_x \times \left(\mathcal T-\bigcup_{y\in A : d(x,y)\leq r} \mathcal T_y\right) \right)\cup\bigcup_{y\in A} \left( \left(\mathcal T-\bigcup_{x\in A : d(x,y)\leq r} \mathcal T_y\right)\times \mathcal T_x\right).$$
Since all $\mathcal T_x$ are closed, $\Delta_r$ is Borel, as required.
\end{proof}

\subsection{Atom-less measures.} \label{subsec:MC}

We now show a result needed in the next section, but not needed for the proof of Theorem \ref{thm:boundarymap}; we include it here since we established the setup for its proof.

The result provides a $\Gamma$-map from atom-less probability measures on $\partial X$ to probability measures on $\partial X^{(3)}$.

Given a hyperbolic space $X$, denote by $\Prob_c(\partial X)$ the set of
	all atom-less probability measures on the standard Borel space $\partial X$.

	\begin{lemma}
	\label{lem:boundary_to_bounded}
	 Let $\Gamma$ be a countable group acting by isometries on the hyperbolic space $X$. Then there is a $\Gamma$-map
	\[
		\Psi:\Prob_c(\partial X)\overto{}\Prob(\partial X^{(3)}).
	\]

 Moreover, for all $\mu,\nu\in \Prob_c(\partial X)$ with $\mu$ absolutely continuous with respect to $\nu$ we have
 $$\left\| \frac{d\Psi(\mu)}{d\Psi(\nu)}\right\|_\infty\leq \left(\left\|\frac{d\mu}{d\nu}\right\|_\infty\right)^3.$$
	\end{lemma}

	\begin{proof}
	We have a $\Gamma$-map
	\[
		\Phi:\Prob_c(\partial X)\overto{}\Prob(\partial X^3),\qquad \mu\mapsto \mu\times\mu\times\mu.
	\]
	In fact, the assumption that $\mu$ has no atoms on a space $\partial X$ implies that $\mu\times\mu\times\mu$
	gives zero mass to the diagonals in $\partial X \times \partial X\times \partial X$, and so is fully supported on $\partial X^{(3)}$.
We thus get a well defined map
\[		\Prob_c(\partial X)\overto{}\Prob(\partial X^{(3)}),\qquad \mu\mapsto \mu\times\mu\times\mu, \]
as required.

 In order to prove the moreover part, let $\mu,\nu\in \Prob_c(\partial X)$ and let $f\in L^\infty(\partial X)$ be the Radon-Nikodym derivative of $\mu$ with respect to $\nu$. Then $(x,y,z)\mapsto f(x)f(y)f(z)$ is the Radon-Nikodym derivative of $\Phi(\mu)$ with respect to $\Phi(\nu)$, so that $\|\frac{d\Phi(\mu)}{d\Phi(\nu)}\|_\infty= (\|\frac{d\mu}{d\nu}\|_\infty)^3$. The measures $\Psi(\mu)$ and $\Psi(\nu)$ are push-forward measures of $\Phi(\mu)$ and $\Phi(\nu)$, and pushforwards do not increase the $L^\infty$ norm of Radon-Nikodym derivatives, so we obtain the required bound.
\end{proof}

\subsection{Proof of Theorem \ref{thm:boundarymap}}
We are now ready to prove Theorem \ref{thm:boundarymap}.
By hypothesis, the group $\Gamma$ is a countable group acting by isometries on the hyperbolic space $X$, without fixing any bounded subset of $X$, any point of $\partial X$ and any pair of points of $\partial X$.

We start with some preliminary claims that we will use a few times.

First of all, we exclude the existence of $\Gamma$-maps with various targets.

\begin{claim}\label{Claim:no-bdd}
We have
\begin{align*}
    \Map_\Gamma(B_-,\Prob(\bar X^b)) 
    & =\Map_\Gamma(B_+,\Prob(\bar X^b))\\
    &=\Map_\Gamma(B_-\times B_+,\Prob(\partial X^{(3)}))\\
    &=\varnothing.
\end{align*}
\end{claim}

\begin{proof}
Recall Lemma \ref{lem:horo_to_sets} and Lemma \ref{lem:triples_to_sets} about the existence of Borel coarsely separable pseudo-metrics on $\bar X^b$ and on $\partial X^{(3)}$ respectively, on which $\Gamma$ acts by isometries with unbounded orbits. Then, in view of Proposition~\ref{prop:maps_to_prob}, all sets of $\Gamma$-maps in the statement need to be empty (note that any pseudo-metric is a coarse metric).
\end{proof}

We shall make use of the following general observation with $\Omega$ being $B_-$, $B_+$, or $B_-\times B_+$. 
Note that these Lebesgue $\Gamma$-spaces are ergodic.

\begin{claim}\label{claim:decomp}
    Let $\Gamma\acts \Omega$ be an ergodic action, and $\Gamma\acts V$ be a Borel action on 
    a Polish space. Suppose $V=V_0\cup V_1$ with $V_0\cap V_1=\varnothing$ be a decomposition
    into $\Gamma$-invariant Borel sets, and suppose $\Map_\Gamma(\Omega,\Prob(V_0))=\varnothing$.
    Then the inclusion $\Prob(V_1)\subset\Prob(V)$ gives an identification
    \[
        \Map_\Gamma(\Omega,\Prob(V))=\Map_\Gamma(\Omega,\Prob(V_1)).
    \]
\end{claim}
 \begin{proof}
     Consider a measurable $\Gamma$-equivariant map $\phi\colon\Omega\overto{}\Prob(V)$. We also define  $f_0,f_1\colon\Omega\to[0,1]$ by setting $f_i(\omega)=\phi(\omega)(V_i)$.
     These are $\Gamma$-invariant measurable functions. By ergodicity $f_i(\omega)=c_i$ constants,
     and $c_i\ge 0$ with $c_0+c_1=1$. Assuming $c_0>0$ one obtains a contradiction, 
     because the normalized
     restriction
     $\phi_0(\omega)=c_0^{-1}\cdot \phi(\omega)|_{V_0}$ is a $\Gamma$-map to $\Prob(V_0)$, 
     which are ruled out by assumption. Hence $c_0=0$
     and $\phi(\omega)$ gives full mass to $V_1$, hence $\phi$ can be considered as a $\Gamma$-map
     $\Omega\overto{}\Prob(V_1)$.
 \end{proof}

\begin{claim}\label{claim:2-diagonal}
    For all $\psi_-\in\Map(B_-,\Prob(\partial X))$ and $\psi_+\in\Map(B_+,\Prob(\partial X))$, 
    the diagonal $\Delta=\setdef{(t,t)}{t\in\partial X}\subset \partial X^2$ satisfies
    \[
        \psi_-(x)\times \psi_+(y)(\Delta)=0
    \]
    for a.e. $(x,y)\in B_-\times B_+$.
\end{claim}
\begin{proof}
    Given a probability measure $\eta\in\Prob(\partial X)$ denote by $\Atm(\eta)$
    the set of atoms of $\eta$; and for $\epsilon>0$ by 
    $\Atm_\epsilon(\eta)=\setdef{a}{\eta(\{a\})\ge \epsilon}$ the
    subset of atoms with weight $\ge \epsilon$.
    The cardinality of $\Atm_\epsilon(\eta)$ is bounded by $\lceil1/\epsilon\rceil$
    and $\Atm(\eta)$ is at most countable.
    Let $\alpha,\beta\in\Prob(\partial X)$ be two probability measures.
    A standard application of Fubini's theorem shows that 
    \[
        \alpha\times\beta(\Delta)=\sum_{t\in \partial X} \alpha(\{t\})\cdot\beta(\{t\})
    \]
    where the non-zero terms correspond to $t\in \Atm(\alpha)\cap \Atm(\beta)$.
    
    For $(x,y)\in B_-\times B_+$, observe that  the sets $\Atm(\psi_-(x))$, $\Atm(\psi_+(y))$ 
    and their intersection vary measurably and $\Gamma$-equivariantly.
    
    Assume, towards contradiction, that $\psi_-(x)\times \psi_+(y)(\Delta)>0$. 
    This is equivalent to asserting that for some $\epsilon>0$, the set
    \[
        \setdef{(x,y)\in B_-\times B_+}{\Atm_\epsilon(\psi_-(x))\cap\Atm_\epsilon(\psi_+(y))\ne \varnothing}
    \]
    has positive measure. Since the $\Gamma$-action on $B_-\times B_+$ is ergodic,
    we deduce that this set is conull,
    and the cardinality of the set $A(x,y)=\Atm_\epsilon(\psi_-(x))\cap\Atm_\epsilon(\psi_+(y))$
    is essentially constant. 
    
    By Fubini's theorem for a.e. $x\in B_-$ we have $A(x,y)=A(x,y')$ for a.e. $y,y'\in B_+$.
    Thus $A(x,-)$ is essentially independent of the second variable.
    For the same reason it is essentially independent of the first variable; hence $A(x,y)=A$
    for a fixed finite subset $A\subset \partial X$.
    Since the map $(x,y)\mapsto A(x,y)=A$ is equivariant, it follows that
    $\Gamma$ has a finite orbit $O$ on $\partial X$. By hypothesis $\Gamma$ does not fix a point or a pair of points in $\partial X$, so that $|O|\geq 3$. By considering the restriction  of the pseudo-metric on $\partial X^{(3)}$ from Lemma~\ref{lem:triples_to_sets} to the finite set of triples of distinct points in $O$, we obtain a $\Gamma$-invariant orbit in $\partial X^{(3)}$, contradicting that the action on $\partial X^{(3)}$ has unbounded orbits.
    
    Hence  $\psi_-(x)\times \psi_+(y)(\Delta)=0$ a.e. on $B_-\times B_+$ as claimed.
\end{proof}

\begin{proof}[End of proof of Theorem~\ref{thm:boundarymap}]
    
Since the actions $\Gamma\acts B_-$, $\Gamma\acts B_+$ are amenable,
the spaces $\Map_\Gamma(B_\pm,\Prob(\bar X))$ are not empty.
Consider the $\Gamma$-invariant Borel decomposition $\bar{X}=\bar{X}^b\cup\bar{X}^u$. 
In view of Claim~\ref{Claim:no-bdd} and Claim~\ref{claim:decomp}
we have
\[
    \Map_\Gamma(B_\pm,\Prob(\bar{X}))=\Map_\Gamma(B_\pm,\Prob(\bar{X}^u))
\]
and post-composing with the $\Gamma$-map $\pi_*:\Prob(\bar X^u)\to \Prob(\partial X)$ (see \S\ref{sec:Gromov-bd}), we obtain
\[
    \Map_\Gamma(B_\pm,\Prob(\partial X))\ne \varnothing.
\]
Let us choose some measurable $\Gamma$-equivariant maps
\[
    \psi_{-}:B_{-}\longrightarrow  \Prob(\partial X),\qquad \psi_{+}:B_{+}\longrightarrow  \Prob(\partial X).
\]

    Consider the partition of $(\partial X)^3$ into $\Gamma$-invariant Borel sets
    \[
        (\partial X)^3=\partial X^{(3)}\cup \Delta_{12}\cup \Delta'_{13}\cup \Delta'_{23}
    \]
    where $\Delta_{ij}=\setdef{(\xi_1,\xi_2,\xi_3)}{\xi_i=\xi_j}$,
    $\Delta_{123}=\setdef{(\xi,\xi,\xi)}{\xi\in\partial X}$, 
    $\Delta'_{ij}=\Delta_{ij}\setminus\Delta_{123}$.
     
    Note that $\Map_\Gamma(B_-\times B_+,\Prob(\partial X^{(3)}))=\varnothing$ by Claim \ref{Claim:no-bdd}.
    Therefore, using Claim~\ref{claim:decomp}, every measurable $\Gamma$-map 
    $\Psi:B_{-}\times B_{+}\longrightarrow  \Prob((\partial X)^3)$ 
    gives full measure to $\Delta_{12}\cup \Delta'_{13}\cup \Delta'_{23}$. 
    We shall apply this to the map
    \[
		\Psi(x,y)=\psi_-(x)\times \psi_+(y)\times \half(\psi_-(x)+\psi_+(y)).
    \] 
    Note that Claim~\ref{claim:2-diagonal} implies that a.e. $\Psi(x,y)(\Delta_{12})=0$.
    
   Let $\Delta$ denote the diagonal in $\partial  X \times \partial X$. The projection $\Delta'_{13}\to \Delta$, $(\xi,\eta,\xi)\mapsto (\xi,\xi)$ shows that
    \[
        \Psi(x,y)(\Delta'_{13})\le\Psi(x,y)(\Delta_{13})=\half\psi_-(x)\times\psi_+(y)(\Delta)
                +\half\psi_-(x)\times\psi_-(x)(\Delta).
    \]
    Moreover, the first term of the right hand side vanishes by Claim~\ref{claim:2-diagonal}, so that $\Psi(x,y)(\Delta'_{13})\le \half\psi_-(x)\times\psi_-(x)(\Delta)$. 
    Similarly, using the projection $\Delta'_{23}\to \Delta$, $(\eta,\xi,\xi)\mapsto (\xi,\xi)$
    gives
    \[
        \Psi(x,y)(\Delta'_{23})\le\Psi(x,y)(\Delta_{23})=\half\psi_+(y)\times\psi_+(y)(\Delta).
    \]
    We conclude that 
    \begin{align*}   
    1 &= \Psi(x,y)(\Delta'_{13})+ \Psi(x,y)(\Delta'_{23}) \\
    &\leq \half\psi_-(x)\times\psi_-(x)(\Delta) + \half\psi_+(y)\times\psi_+(y)(\Delta) \\
    &\leq 1,
    \end{align*}
    so that 
    $\psi_-(x)\times\psi_-(y)(\Delta)=\psi_+(y)\times\psi_+(y)(\Delta)=1$.
    Recall that, by Fubini's theorem,  for any probability measure $\nu$ one has $\nu\times\nu(\Delta)=1$
    if and only if $\nu$ is a Dirac measure $\nu=\delta_z$ for some $z$.
    Hence there exist maps  $\phi_\pm:B_\pm\longrightarrow  \partial X$ so that
     \[
        \psi_-(x)=\delta_{\phi_-(x)},\qquad \psi_+(y)=\delta_{\phi_+(y)}.
    \]
    The maps $\phi_\pm:B_{\pm}\overto{}\partial X$ are measurable and $\Gamma$-equivariant 
    because the maps $\psi_\pm:B_\pm\longrightarrow  \Prob(\partial X)$ are.

    We also note that,  since $\Psi(\Delta_{12})=0$ a.e. we have 
    $\phi_-(x)\ne \phi_+(y)$  for a.e. $(x,y)\in B_-\times B_+$.
    Therefore the map $\phi_{\bowtie}=\phi_{-}\times\phi_+$ can be viewed
    as a measurable $\Gamma$-equivariant map $B_-\times B_+\overto{}\partial X^{(2)}$
    into distinct pairs.
    
Next, we show that $\phi_\pm$ are essentially unique. 
Indeed, we showed that every $\psi_\pm\in\Map_\Gamma(B_\pm,\Prob(\partial X))$
take values in Dirac measures. Assume we have two pairs of $\Gamma$-maps 
$\psi^i_\pm:B_\pm\to\Prob(\partial X)$ with $i=1,2$. 
Consider the $\Gamma$-maps $B_{\pm}\longrightarrow  \Prob(\partial X)$ given by 
$z\mapsto \half(\psi^1_\pm(z)+\psi^2_\pm(z))$.
Such a map has the form $z\mapsto \delta_{\phi_\pm(z)}$ only if
$\psi^1_\pm(z)=\psi^2_\pm(z)=\delta_{\phi_\pm(z)}$.

Next consider an arbitrary map $\theta\in \Map_\Gamma(B_-\times B_+,\partial X)$, consider 
the measurable $\Gamma$-equivariant map  be defined by 
\[
    \Theta:B_-\times B_+\overto{}(\partial X)^3,
    \qquad\Theta(x,y)=(\theta(x,y),\phi_-(x),\phi_+(y)).
\] 
Using the $\Gamma$-equivariant Borel decomposition 
$(\partial X)^3=\partial X^{(3)}\cup \Delta'_{12}\cup \Delta'_{13}\cup\Delta_{23}$ 
and ergodicity of the $\Gamma$-action on $B_-\times B_+$
we deduce that $\Theta(x,y)$ lies in one of these sets for a.e. $(x,y)$. 
Lemma \ref{lem:triples_to_sets} combined with coarse metric ergodicity of $B_-\times B_+$ rules out the set of distinct triples $\partial X^{(3)}$.
Since the $23$-components of $\Theta$ is just $\phi_{\bowtie}$,
and the latter take values in the space of distinct pairs $\partial X^{(2)}$,
rule out $\Delta_{23}$. 
It follows that, up to a null set, either $\Theta(x,y)\in \Delta'_{12}$ or $\Theta(x,y)\in \Delta'_{13}$.
This corresponds to $\theta=\psi_-\circ \pr_-$ or $\theta=\psi_+\circ\pr_+$ as claimed in (ii).

Finally, notice that (iii) can be deduced from (ii) by looking at the coordinates in $(\partial X)^{(2)}$.
\end{proof}


%

\section{Classification of actions on hyperbolic spaces}
\label{sec:proof_main}

In this section we prove Theorem \ref{T:main_intro}.

\subsection{Lattices acting on hyperbolic spaces}

Recall that a subset $A$ of a metric space is \textit{coarsely dense} if there exists a constant $R$ such that $X$ is the $R$-neighborhood of $A$.  Recall moreover that given a group $\Gamma$ acting on a hyperbolic space $X$, its limit set is the set of boundary points that are equivalence classes of sequence of points $\gamma x$, for some fixed $x\in X$.

\begin{definition}
\label{defn:coarsely_min}
 We say that an action $\Gamma\acts X$ on a hyperbolic space is \emph{coarsely minimal} if $X$ is unbounded, the limit set of $\Gamma$ in $\partial X$ is not a single point, and every quasi-convex $\Gamma$-invariant subset of $X$ is coarsely dense.
\end{definition}

Note that coarse minimality is a stronger requirement than asking that the orbits of $\Gamma$ have full limit set. Indeed, consider for example, the free group $F_2$ acting  on its Cayley tree $T$,  which is the universal cover of a graph with $1$ vertex and two loops. Glueing, for each $n$, an interval of length $n$ to that vertex, one obtains a locally infinite graph whose universal cover is a tree $\tilde T$ on which $F_2$ acts with full limit set. The tree $T$ embeds as a convex $F_2$-invariant subtree of  $\tilde T$, but $T$ is not coarsely dense in $\tilde T$, so that the $F_2$-action on $\tilde T$ is not coarsely minimal. 

Notice that if $H$ is an infinite normal subgroup of infinite index of the hyperbolic $G$, then the action of $H$ on a Cayley graph of $G$ is coarsely minimal, but not cobounded.

Given a metric space $X$ and $C\geq0$, denote by $\bdd_C(X)$ the set of all closed subsets of diameter at most $C$, endowed with the Hausdorff metric. Notice that $\bdd_C(X)$ is quasi-isometric to $X$.

\begin{definition}
\label{defn:equivalence}
 Two actions $\Gamma\acts X_1,X_2$ on metric spaces $X_1,X_2$ are \emph{equivalent} if there exists an equivariant quasi-isometry $X_1\to \bdd_{C}(X_2)$ for some $C\geq 0$.
\end{definition}

The reason for having $\bdd_{C}(X_2)$ instead of $X_2$ is that we want to allow the situation where some group element has a fixed point in $X_1$ but merely a bounded orbit in $X_2$; for example, we want to declare all actions on bounded metric spaces to be equivalent.

We note that the definition above is very similar to \cite[Definition 3.5]{ABO}, and in fact the two definitions are easily seen to be equivalent for cobounded actions.

\begin{remark}\label{rem:minimal}
Consider an action $\Gamma\acts X$ on a geodesic hyperbolic space.
 \begin{enumerate}
  \item If the action is cobounded, then it is coarsely minimal.

  The following two items follow from a construction well-known to experts, namely taking the  coarse convex hull of an orbit and approximating it with a graph; this is explained for example in \cite[Remark 4]{GST}.
  
  \item If the limit set of $\Gamma$ is not a single point, then there is a coarsely minimal action $\Gamma\acts Y$ on a geodesic hyperbolic space $Y$ and an equivariant quasi-isometric embedding $Y\to X$.
  
  \item If $\Gamma$ is countable and $\Gamma\acts X$ is coarsely minimal, then $\Gamma\acts X$ is equivalent to an action on a separable geodesic hyperbolic space (in fact, a graph).
 \end{enumerate}
\end{remark}

Consider a semisimple Lie group $G=G_1\times\cdots\times G_N$ without compact factors. Assume that either $N\geq 2$ or $N=1$ and $G=G_1$ has rank at least 2.

Re-order the factors in such a way that $G_i$ is rank-one if and only if $1\le i\le n$, for some $n\leq N$.

We now re-state our main theorem.

\begin{theorem}\label{T:main}
	Let $\Gamma$ be an irreducible lattice in $G=G_1\times\cdots\times G_N$ as above.	
	Then every coarsely minimal action of $\Gamma$ on a geodesic hyperbolic space is equivalent to the action
	\[
		\Gamma\longrightarrow  G\overto{\pr_i} G_i\overto{}\Isom(X_i,d_i)
	\]
	for some $i \in \{1, \dots, n\}$, where  $X_i$ is the symmetric space for $G_i$.
\end{theorem}

Let us extend the notation of the theorem by denoting $X_j$, for $j>n$, the symmetric space for $G_j$.

 Let $\Gamma\acts (X,d)$ be a coarsely minimal action of $\Gamma$ on the geodesic hyperbolic space $X$. Denote by $d_i$ the pseudo-metric on $\Gamma$ corresponding to $\Gamma\acts X_i$ (with respect to some basepoint $x_i$, for $i\leq n$), and $d$ the pseudo-metric corresponding to $\Gamma\acts (X,d)$.

\subsection{Ruling out elementary actions}

In order to be able to apply Theorem~\ref{thm:boundarymap}, we have to rule out that $\Gamma$ fixes a pair of points in $\partial X$ (the case that it fixes one point being ruled out by hypothesis). If that were the case, the subgroup $\Gamma'$ of index at most 2 of $\Gamma$ that fixes a boundary point would admit the quasimorphism described in \cite[Proposition 3.7]{CCMT}. According to \cite{BurgerMonod1,BurgerMonod2}, $\Gamma'$ does not admit unbounded quasimorphisms, so that according to \cite[Lemma 3.8]{CCMT} the action $\Gamma\acts X$ has a single limit point, contradicting the coarse minimality of the action.
From now on, we will assume that $\Gamma$ does not fix a point or a pair of points in $X$.

Finally, we can assume that $X$ is separable by Remark~\ref{rem:minimal}.

\subsection{Boundary map from one factor}


 For every $1\leq i \leq N$, we let $B_i=G_i/P_i$ and we let $\nu_i$ be a measure in the Haar class on $B_i$, and similarly we denote $B=G/P=B_1\times \dots\times B_N$. Note that we are in the setting of Proposition \ref{prop:G_P_coarse_ergodic} about coarse metric ergodicity.

Theorem \ref{thm:boundarymap} affords now two $\Gamma$-maps $B\to \partial X$ satisfying various properties. The first of these properties implies that these two maps must coincide a.e. We assume henceforth that both maps are identical, and we denote it by $\phi \colon B\to \partial X$. 

\begin{claim}\label{claim:factors}
 The map $\phi$ factors through one of $B_i$. In other words,  there is an index $i\in \{1,\dots,N\}$ and a $\Gamma$-map $B_i\overto{\phi_i} \partial X$ such that 
	\[
		\phi\colon B\overto{\pr_i} B_i\overto{\phi_i} \partial X.
	\]
\end{claim}

\begin{proof}
The equivariant map $\phi:B\overto{}\partial X$ cannot be essentially constant, because $\Gamma$ does not have a fixed point in $\partial X$.
Fix an $i\in\{1,\dots,N\}$ so that $\phi$ is not essentially constant over the $B_i$-factor of $B=B_1\times \dots \times B_N$.
We let $B'$ be the product of the other factors, and identify $B\simeq B_i\times B'$.
Using this identification we consider the map
\[
	\Phi:B_i\times B' \times B_i\times B' \to \partial X^2,\quad
	(x,y,x',y') \mapsto (\phi(x,y),\phi(x,y')).
\]
By Theorem~\ref{thm:boundarymap}(iii) we have three cases: $\Phi(B\times B)$ is contained in the diagonal $\Delta\subset \partial X^2$,
$\Phi=\phi_{\bowtie}$, or $\Phi=\tau\circ\phi_{\bowtie}$, where $\phi_{\bowtie}=\phi\times \phi$ and $\tau(\xi,\eta)=(\eta,\xi)$.
In the first case we see that $\phi(x,y)$ is independent of $y\in B'$, and therefore descends to a $\Gamma$-map $B_i\to \partial X$ as required. 
In the second and third cases, $\phi$ is independent of $x\in B_i$, contradicting our choice of $i$.
\end{proof}

From now on we fix the index $i \in \{1, \dots, N\}$ afforded by Claim~\ref{claim:factors}.

\subsection{The factor $B_i$ is associated with a rank-one factor} 
	\label{sub:step_2.5_rank_1_factor}\hfill{}\\

We now explain that the factor $B_i$ afforded by Claim~\ref{claim:factors}
is associated with a rank-one factor, that is $i\leq n$.
We argue by contradiction, assuming $i>n$, that is, $G_i$ is of higher rank.

As before, we let $P_i<G_i$ be a minimal parabolic.
In the sequel we will omit the index $i$ and denote $P=P_i$.
Let $A<P$ be a maximal split torus.
We let $W=N_{G_i}(A)/Z_{G_i}(A)$ be the corresponding Weyl group and let $S\subset W$ be the standard Coxeter generators associated with the positivity defined by $P$.
Letting $Z=Z_{G_i}(A)$ be the centralizer of $A$, we note that $W$ acts naturally on $G_i/Z$ by $G_i$-automorphisms.

As usual we identify the set $S$ with the set of simple roots of $G_i$ associated with the pair $(A,P)$. 
Any subset $T$ of $S$ generates a subgroup $W_T<W$
and it corresponds to a standard parabolic $P_T<G_i$ containing $P=P_\varnothing$.
All the subgroups of $G_i$ containing $P$ are of this form.
Denoting by $\pi_T:G_i/Z\to G_i/P_T$ the standard map coming from the inclusion $Z<P<P_T$, we note that 
\[ W_T=\{w\in W\mid \pi_T\circ w=\pi_T \}. \]

We let $\pi=\pi_\varnothing:G_i/Z\to G_i/P$ be the standard map and let $w_0\in W$ be the longest element (with respect to the word distance induced by $S$).
It is a standard fact that map $\pi\times \pi\circ w_0:G_i/Z\to G_i/P\times G_i/P$
is injective and its image is Zariski open
(this image is the big cell in the Bruhat decomposition of $G_i/P\times G_i/P$). 


We also endow $G_i/Z$ with the Haar measure class. By the preceding paragraph, we may identify it, as measured $G_i$-spaces, with $B_i\times B_i=G_i/P\times G_i/P$ via the map
$\pi\times \pi\circ w_0$.

We set $\phi_i:B_i\to \partial X$ to be the map given in Claim~\ref{claim:factors}.
Note that $\phi_i$ is not essentially constant, as $\partial X$ has no $\Gamma$-fixed points.
We thus may find a bounded measurable function $f_0:\partial X \to \mathbb{C}$ such that $f_0\circ \phi_i$ is not essentially constant.
We fix such a function $f_0$.

We consider the map $\psi=\phi_i\circ p_1:B_i\times B_i\longrightarrow  \partial X$,
where $p_1:B_i\times B_i\longrightarrow  B_i$ is the projection on the first factor. Under our identifications $B_i = G_i/P$ and $B_i\times B_i = G_i/Z$, the projection $p_1$ is identified with the map $
\pi = \pi_\varnothing \colon  G_i/Z \to G_i/P$. We set
\[ U=\{w\in W\mid \psi\circ w \mbox{ agrees a.e. with }\psi\}<W. \]
By Theorem~\ref{thm:boundarymap}(ii) we conclude that $U<W$ is of index at most 2.

We shall now show that $U$ is contained in a proper parabolic subgroup $W_T$ of $W$. 
Consider now the algebra $L^\infty (G_i/Z)$ and its subalgebra $\pi^*(L^\infty (G_i/P))$ consisting of functions
pulled back from $L^\infty (G_i/P)$ under $\pi:G_i/Z\to G_i/P$ (which we identify with $p_1:B_i\times B_i\to B_i$).
Consider the subalgebra
\[ \{f\in L^\infty (G_i/Z) \mid f\in L^\infty (G_i/P) \mbox{ and for every }
u\in U,~f\circ u \mbox{ agrees a.e. with } f\}. \]
This is a weak*-closed $G_i$-invariant subalgebra of $L^\infty (G_i/Z)$.
By Mackey's point realization theorem this algebra coincides with the subalgebra of functions pulled back from a $G_i$-factor of $G_i/Z$.
As all functions in it are pulled back from $G_i/P$, this factor is of the form $G_i/P_T$ for some $T\subset S$. 
As the algebra includes the non-constant function $f_0\circ \psi$,
we conclude that $P_T\neq G_i$, thus $T\neq S$ and $W_T\neq W$.
We have that $\pi_T\circ u=\pi_T$ for every $u\in U$, thus $U<W_T$.
It follows that $W_T$ is of index~$2$ in $W$ and in particular it is a normal subgroup.

To obtain a contradiction and conclude this step of the proof, we use that $W$ is an irreducible finite Coxeter group with $|W|>2$, since $G_i$ is simple of  rank~$\geq 2$. Therefore, there is $t \in T$ and $s \in S \setminus T$ such that $s$ and $t$ do not commute. In particular $sts$ is a reduced word in $W$. Since $W_T$ has index~$2$ we have $sts \in W_T$. By \cite[Cor.~1 in Chapter~IV, \S 1, no.~1.8]{Bourbaki}, we deduce that $s \in T$, a contradiction.

\subsection{Bounded in $G_i$ is $d$-bounded} 
	\label{sub:step_2_bounded_to_bounded}\hfill{}\\

	Next, we show that $d$ is ``smaller'' than $d_i$, for $i$ as in Claim \ref{claim:factors}.
	
	\begin{claim}\label{C:d-prec-di}
		There exist $L,C$ so that for all $\gamma,\gamma'$ we have
		\[
			d(\gamma,\gamma')\le L\cdot d_i(\gamma,\gamma')+C.
		\]
	\end{claim}

 \begin{proof}
	As $G$ is of higher rank, we deduce from \S\ref{sub:step_2.5_rank_1_factor} that $N>1$ and in particular, we get that $\pr_i(\Gamma)$ is dense in $G_i$. Hence, in the metric $d_i$ any pair of points is connected by a $(1,1)$-quasi-geodesic, and to prove the claim
	it suffices to show that sequences that are bounded in $G_i$ are bounded in $(\Gamma,d)$.
	In other words, it suffices to show that any sequence $\{\gamma_j\}$ in $\Gamma$ for which
	$\{\pr_i(\gamma_j)\}$ is precompact in $G_i$, one has
	\[
		\sup_j d(\gamma_j,1)<+\infty.
	\]

Let $\mu=\phi_*\nu_i\in\Prob(\partial X)$ be the pushforward of the measure $\nu_i$ on $B_i$.
By the metric ergodicity of $B_i$, $\mu$ has no atomic part.
Indeed, if it had we would get a countable invariant subset of $\partial X$ and upon endowing it with the discrete metric, in view of the metric ergodicity of $B_i$, we will conclude that this set contains a single point which is $\Gamma$-invariant, contradicting the hypothesis.

By Lemma \ref{lem:boundary_to_bounded}, we have a $\Gamma$-map
	\[
		\Psi:\Prob_c(\partial X)\overto{}\Prob(\partial X^{(3)}),
	\]
	where, as before, $\Prob_c(\partial X)$ is the set of
	all atom-less probability measures on $\partial X$.
	
		We assume that $\{\gamma_j\}$ is such that $\{\pr_i(\gamma_j)\}$ is precompact in $G_i$.
We now need the following:

\begin{lemma}
 \label{lem:RN}
  Given any precompact subset $\{\gamma_j\}\subseteq G_i$, the corresponding Radon-Nikodym derivatives are uniformly bounded, that is, we have
 	\[
 		\sup_{j\ge 1} \|\frac{d\gamma_j\nu_i}{d\nu_i}\|_\infty<+\infty.
 	\]
 \end{lemma}

 \begin{proof}
Since $\nu_i$ is the Patterson--Sullivan measure for $G_i$, we have for any $g\in G_i$ and a.e. $\xi\in \partial X_i$
\[
    \frac{d g\nu_i}{d\nu_i}(\xi)=e^{-\delta_i \beta(g,\xi)}
\]
where $\delta_i$ is a constant associated with the symmetric space $X_i$ for $G_i$, and $\beta(g,\xi)$
is the Busemann function. Since the Busemann function is bounded
by the distance in $X_i$, $|\beta(g,\xi)|\le d_i(go,o)$, the Lemma is proved. 
 \end{proof}

Hence, by Lemma \ref{lem:RN} we have:
	\[
		\sup_{j\ge 1}\ \|\frac{d\gamma_j\nu_i}{d\nu_i}\|_\infty=C<+\infty.
	\]
	We thus have the same bound on the Radon-Nikodym derivatives of $\mu$
	\[
		\sup_{j\ge 1}\ \|\frac{d\gamma_j\mu}{d\mu}\|_\infty\le C
	\]
	and also, by the moreover part of Lemma \ref{lem:boundary_to_bounded},
	\[
		\sup_{j\ge 1}\ \|\frac{d\Psi(\gamma_j\mu)}{d\Psi(\mu)}\|_\infty\le C^3.
	\]
	Setting $U=\Prob(\partial X^{(3)})$, recall the maps $\beta_\epsilon:\Prob(\partial X^{(3)})\to \bdd(\partial X^{(3)})$ from \S\ref{sec:bary}, which are $\Gamma$-equivariant (as they are equivariant with respect to all isometries). Fix $\epsilon\in (0,1/2)$ and set $R=r_{\epsilon/C}(\Psi(\mu))+1$.
 Fixing $j$ and applying Lemma~\ref{L:centers} with $m_1=\Psi(\mu)$
 and $m_2=\Psi(\gamma_j\mu)$ we get 
 \[
		\gamma_j (\beta_{\epsilon}\circ \Psi(\mu))=\beta_{\epsilon}\circ\Psi(\gamma_j\mu) \subset N_R\big(\beta_{\epsilon/C}(\Psi(\mu))\big).
	\]
 Since $N_R\big(\beta_{\epsilon/C}(\Psi(\mu))\big)$ is a bounded subset of $\partial X^{(3)}$, we get that the sequence $\{\gamma_j\}$ has bounded orbits in $\partial X^{(3)}$, whence it is bounded in $(\Gamma,d)$ by the ``moreover'' part of Lemma \ref{lem:triples_to_sets}.
\end{proof}

\subsection{Unbounded in $G_i$ is $d$-unbounded} 
\label{sub:unbounded_in_g_i_is_unbounded}

\begin{claim}\label{C:d_i-prec-d}
	There exist $L,C$ so that for all $\gamma,\gamma'$ we have
	\[
		d_i(\gamma,\gamma')\le L\cdot d(\gamma,\gamma')+C.
	\]
\end{claim}

\begin{proof}
Let $\gamma_0\in \Gamma$ be a loxodromic element for $\Gamma\acts X$, which exists by Gromov's classification of actions on hyperbolic spaces \cite[Section 8]{Gromov}. Since the identity map $(\Gamma,d_i)\to(\Gamma,d)$ is coarsely Lipschitz, $\gamma_0$ is loxodromic for $\Gamma\acts X_i$ as well. Let us choose a constants $A$ so that for each $j\geq 0$ we have $j/A\leq d_i(1,\gamma_0^j)\leq Aj$ and $j/A\leq d(1,\gamma_0^j)\leq Aj$.
Note that, up to increasing $A$, since $G_i$ is rank-one and $\gamma_0$ is loxodromic, the set 
\[ \{\kappa\gamma_0^j\mid \kappa\in K_i,~j\geq 0\} \]
is $A$-dense in $X_i$, where we denote by $K_i$ a maximal compact subgroup of $G_i$. This is because, since $K_i$ acts transitively on $\partial X$, there exists a constant $C$ such that given any two points on a sphere around $x_i$, there is an element of $K_i$ that moves the first point $C$-close to the second one.

Since the rank of $G_i$ equals one, we know by hypothesis that $N>1$ and in particular, we get that $\pr_i(\Gamma)$ is dense in $G_i$.
Approximating elements of $K_i$ by elements of $\Gamma$ we get that the set 
\[ \{\kappa\gamma_0^j\mid d_i(1,\kappa)\leq 1,~j\geq 0\} \]
is $A+1$-dense in $\Gamma$ for the metric $d_i$. 
Enlarging $A$ if necessary, let us further assume that $d(1,\gamma)\leq Ad_i(1,\gamma)+A$ for all $\gamma\in\Gamma$.

Consider an arbitrary $\gamma\in \Gamma$. 
Using the above we find $\kappa\in \Gamma$ such that $d_i(1,\kappa)\leq1$ and $d_i(\kappa\gamma,\gamma_0^j)\leq A+1$ for some $j\geq 0$. Therefore, we obtain
\begin{align*}
 d_i(1,\gamma) &= d_i(\kappa,\kappa\gamma)\\
 &\leq d_i(1,\gamma_0^j)+A+2\\
 & \leq Aj+A+2\\
 & \leq A^2d(1,\gamma_0^j)+A+2\\
 &\leq  A^2 (d(1,\kappa)+d(\kappa,\kappa\gamma)+d(\kappa\gamma,\gamma_0^j))+A+2\\
 & \leq A^2d(1,\gamma)+A^3+2A^2+3A+2,
\end{align*}

as required.
\end{proof}

\subsection{Conclusion}

We have seen in Section~\ref{sub:step_2.5_rank_1_factor} that the rank of $G_i$ is~$1$. 
Claims \ref{C:d-prec-di} and \ref{C:d_i-prec-d} imply that there is a $\Gamma$-equivariant quasi-isometric embedding $f:X_i\to \bdd_C(X)$ for some $C\geq 0$ (since $(\Gamma,d_i)$ is $\Gamma$-equivariantly quasi-isometric to $X_i$, and there is a $\Gamma$-equivariant quasi-isometric embedding $(\Gamma,d_i)\to X$). 
Since $X_i$ is hyperbolic, the image of that quasi-isometric embedding is quasi-convex. Since the $\Gamma$-action on $X$ is coarsely minimal, it follows that the $\Gamma$-action on $X$ is cobounded, so that the quasi-isometric embedding $(\Gamma,d_i)\to X$ is in fact a quasi-isometry. This proves that $\Gamma\acts X_i$ is indeed equivalent to $\Gamma\acts X$, as required.\qed

\bibliographystyle{alpha}
\bibliography{biblio}

\end{document}